\documentclass[12pt, reqno]{amsart}
\usepackage[utf8]{inputenc}
\usepackage[english]{babel}

\usepackage{ifthen}
\usepackage{microtype}
\usepackage[dvipsnames]{xcolor}
\usepackage{comment}
\usepackage{blkarray}
\usepackage{fullpage}
\usepackage{mathtools}
\usepackage{xfrac}
\usepackage{centernot}
\usepackage{enumitem}

\usepackage{mathdots}
\usepackage{multirow}
\usepackage{amssymb}
\usepackage{soul}

\newlist{paraenum}{enumerate}{1}
\setlist[paraenum]{wide, label=(\arabic*)}

\makeatletter
\patchcmd{\@setaddresses}{\indent}{\noindent}{}{}
\patchcmd{\@setaddresses}{\indent}{\noindent}{}{}
\patchcmd{\@setaddresses}{\indent}{\noindent}{}{}
\patchcmd{\@setaddresses}{\indent}{\noindent}{}{}
\makeatother

\usepackage{caption}
\usepackage{subcaption}
\usepackage{float}

\usepackage{nicematrix}
\usepackage{tikz}    
\usetikzlibrary{cd}
\usetikzlibrary{backgrounds}
\usetikzlibrary{shapes}
\usetikzlibrary{arrows}
\usetikzlibrary{positioning}
\usetikzlibrary{calc}
\usetikzlibrary{quotes}
\usetikzlibrary{fit}
\usetikzlibrary{decorations.pathreplacing}

\usepackage[pdftex]{hyperref}
\hypersetup{
    colorlinks=true,
    linkcolor={magenta},
    citecolor={blue},
    urlcolor={blue}
}
\usepackage[nameinlink,capitalise]{cleveref}

\usepackage[
    backend=bibtex,
    style=numeric, 
    sorting=nyt, 
    maxbibnames=5, 
    giveninits=true, 
    hyperref=true, 
    url=false, 
    doi=false 
]{biblatex}
\numberwithin{equation}{section}
\theoremstyle{plain}
\newtheorem{theorem}{Theorem}[section]
\newtheorem{conjecture}[theorem]{Conjecture}
\newtheorem{corollary}[theorem]{Corollary}
\newtheorem{lemma}[theorem]{Lemma}
\newtheorem{proposition}[theorem]{Proposition}
\newtheorem{problem}[theorem]{Problem}

\theoremstyle{definition}

\newtheorem{examplex}[theorem]{Example}
\newenvironment{example}
  {\pushQED{\qed}\examplex}
  {\popQED\endexamplex}

\newtheorem{definition}[theorem]{Definition}

\newcommand{\cg}{\mathcal{G}}

\author{Hannah Göbel}
\address{Technische Universit\"at M\"unchen, 85748 Garching b. München, Boltzmannstr. 3.,  Germany}
\email{hannah\_goebel@web.de}

\author{Pratik Misra}
\address{Technische Universit\"at M\"unchen, 85748 Garching b. München, Boltzmannstr. 3.,  Germany}
\email{pratik.misra@tum.de}

\title[Linear relations of colored Gaussian cycles]%
{Linear relations of colored Gaussian cycles}


\begin{document}

\begin{abstract}
A colored Gaussian graphical model is a linear concentration model in which equalities among the concentrations are specified by a coloring of an underlying graph. Marigliano and Davies conjectured that every linear binomial that appears in the vanishing ideal of a undirected colored cycle corresponds to a graph symmetry. We prove this conjecture for $3,5,$ and $7$ cycles and disprove it for colored cycles of any other length. We construct the counterexamples by proving the fact that the determinant of the concentration matrices of two colored paths can be equal even when they are not identical or reflection of each other. We also explore the potential strengthening of the conjecture and prove a revised version of the conjecture.
\end{abstract}

\maketitle

\section{Introduction}

Graphical models are multivariate statistical models where the conditional independence relations among the variables are encoded by a graph. Different types of graphs (undirected, directed, mixed) can be used to encode these relations. For an undirected graph $G$ with vertex set $[n]=\{1,2,.\ldots,n\}$ and edge set $E$, the associated Gaussian graphical model is a linear concentration model where the entries of the concentration matrices corresponding to the non edges of $G$ are zero. These models are widely used throughout computational biology to model gene interactions \cite{TohHorimoto2002} and in environmental psychology to model community attitudes towards sustainable behaviors \cite{BMS+2019}.

As Gaussian graphical models are parametrized models, a natural question of interest is to understand the defining equations of the model by using the structure of the graph. For instance, in \cite{GeigerMeekSturmfels2006}, \cite{MisraSullivant2021}, and \cite{MisraSullivant2022}, the authors characterize the graphs for which the corresponding defining equations of the model form a toric ideal. However, these computations become practically intractable when the size of the graph increases. Thus, it can be useful to introduce additional symmetries in the concentrations when certain partial correlations interact in the same way \cite{WitAbbruzzo2015}.

In \cite{HojsgaardLauritzen2008}, Hojsgaard and Lauritzen introduced different types of graphical models where they included additional symmetries among the parameters. These symmetries are encoded by coloring the edges and vertices of the graph, i.e., we set two entries of the concentration matrix equal to one another if their corresponding edges or vertices have the same color. These models are used to study gene regulatory networks wherein one imposes symmetries among genes with similar expression patterns \cite{TohHorimoto2002, Vinciotti+2016}. From a computational perspective, introducing the symmetries reduces the dimension of the model, which in turn makes certain computations more feasible. 
 
In \cite{SturmfelsUhler2010}, Sturmfels and Uhler did a computational study on the algebraic properties of colored Gaussian graphical models whose underlying graph is the $4$-cycle. In \cite{Coons+2023}, the authors characterized the graphs and their coloring for which the vanishing ideal were toric. In this paper, we focus on the work done by Marigliano and Davies in \cite{DaviesMarigliano:2021}, where they study the linear relations of the colored Gaussian cycles. In particular, they proved that a linear binomial lies in the vanishing ideal of a uniformly colored Gaussian cycle if and only if there exists a corresponding graph symmetry, and conjectured it to be true for any arbitrary coloring. We prove that the conjecture is true for $3,5$ and $7$-cycles and show that it is not true in general by constructing counterexamples of colored cycles of both even and odd length. In particular, the main results of the paper can be summarized by the following:

\begin{theorem}
Let $\cg$ be a colored $n$-cycle. If $n$ is $3,5,$ or $7$, then every linear binomial in the vanishing ideal $I_\cg$ corresponds to a graph symmetry. However, if $n$ is $4,6,8$ or larger, then there can exist colored cycles whose vanishing ideal contain linear binomials that do not correspond to any graph symmetry.  
\end{theorem}

The construction of these counterexamples depends on the various types of non-trivial path coloring whose concentration matrices end up in having the same determinant. We show that two colored paths can have the same determinant even though they are not identical nor reflection of each other.
We also analyze the potential ways to strengthen the conditions of the conjecture and study if it still holds. 

The paper is organized as follows: We start with some preliminaries on Gaussian graphical models and graph symmetries in Section \ref{Section:Preliminaries}. We also state the conjecture and some existing results that will be used in the later sections. The analysis of the conjecture relies heavily on Theorem 1 of \cite{JonesWest:2005}, which allows us to view the covariance between $i$ and $j$ as the sum over the two paths between them in any cycle. Thus, Section~\ref{section:analysis of path graphs} is dedicated to this path analysis, where we first derive a formula for the determinant of the concentration matrix of a path in terms of the disjoint edge sets of the path. We then focus on identifying the color conditions under which the determinant of the concentration matrices of two paths are equal. 
We prove that the conjecture is true for $3,5,$ and $7$-cycles and end the section by finding two non-trivial color configurations each for paths with even and odd number of vertices, such that the determinant of their corresponding concentration matrices become equal.
Using the non-trivial path color configurations from Section \ref{section:analysis of path graphs}, we construct counterexamples to Conjecture~\ref{Conj:ifandonlyif} of both even and odd length in Section \ref{section:counterexamples and strengthening}. 
We also investigate the possible ways to strengthen the conditions of the conjecture. In particular, we find counterexamples which shows that the conjecture does not hold for uniform vertex coloring as well. However, this also leads us to state and prove a revised version of the conjecture, which is as follows:
\begin{theorem}[The revised conjecture]\label{theorem:revisedConjecture}
Let $\cg$ be any colored $n$-cycle with uniform edge coloring, where $n$ is odd. Then a linear binomial lies in $I_\cg$ if and only if there is a corresponding symmetry in $\cg$. 
\end{theorem}

We end the paper with Section~\ref{Sec:Discussion and open problems}, where we continue to look for other potential non-trivial path color configurations apart from the ones obtained in Section \ref{section:analysis of path graphs}. We conjecture that there can be no other non-trivial path color configurations and end the section by posing a general version of the problem.


\section{Preliminaries}\label{Section:Preliminaries}

\textbf{Gaussian graphical models}: Let \(X = (X_1, X_2,..., X_n) \in \mathbb{R}^n\)\ be a random vector, which is distributed according to a multivariate Gaussian, \(X \sim \mathcal{N}(\mu, \Sigma)\), where $\mu = (\mu_1, \mu_2, \dots, \mu_n) \in \mathbb{R}^n$ is the vector of means and \(\Sigma \in \mathbb{R}^{n\times n}\) is the symmetric, positive definite covariance matrix.
To capture the conditional dependencies among the components of \(X\), we define an undirected graph \( G = (V, E) \) with vertex set \( V = \{1, 2, \dots, n\} \) and edge set $E$, where each vertex \(i\) corresponds to a variable \(X_i\). An edge is placed between vertices \( i \) and \( j \) if and only if the corresponding random variables \( X_i \) and \( X_j \) are dependent. Conversely, if \( X_i \) and \( X_j \) are conditionally independent given all other variables, then there is no edge between \( i \) and \( j \). Throughout the paper, we assume that the edge set $E$ contains self loops at the vertices but no other loops or multi-edges.

\begin{definition}
Consider the linear space $\mathcal{L}_G$ of symmetric matrices $K=(k_{ij})$ in $\mathbb{R}^{n\times n}$ that satisfy the constraint that $k_{ij}=0$ if $\{i,j\}$ is not an edge in $G$. The \textit{Gaussian graphical model} $\mathcal{M}_G$ defined by the graph $G$ is the set of all multivariate Gaussian distributions on random variables $X_1,\ldots,X_n$ with mean $0$, whose concentration (inverse covariance) matrix $K$ lies in $\mathcal{L}_G$. The inverse linear space, $\mathcal{L}^{-1}_G=\overline{\{\Sigma \in \mathbb{R}^{n\times n} : \Sigma^{-1} \in \mathcal{L}_G \}}$ is the Zariski closure of the 
set of covariance matrices for $\mathcal{M}_G$.
\end{definition}

The entries of the positive definite covariance matrix $\Sigma \in \mathcal{L}^{-1}_G$ are the covariances between the random variables, i.e., $\sigma_{ij}=Cov(X_i,X_j)$. Similarly, any concentration matrix $K$ of a distribution in the model $\mathcal{M}_G$ is a positive definite matrix which can be written as a linear combination of linearly independent symmetric matrices,
\[
K = \sum_{\{i,j\} \in E} K_{ij} + \sum_{i \in V} K_{ii},
\]
where each $K_{ij}$ is the matrix with $k_{ij}$ in the $(i,j)^{th}$ and $(j,i)^{th}$ position and $0$ elsewhere. Any statistical model where the concentration matrix can be written as such a linear combination is also called a linear concentration model.
The entries of $K$ are called the partial correlations of the distribution and are useful for understanding the conditional independence constraints on the model. In particular, $k_{ij}=0$ implies that $X_i$ and $X_j$ are independent given all other random variables. 

Notation: For a given graph $G$, we denote $\Sigma$ as the generic $n\times n$ covariance matrix with entries $\sigma_{ij}$ and $K$ as the generic concentration of $G$ in $\mathcal{L}_G$ with entries $k_{ij}$. These variables satisfy the constraints imposed by the symmetry and positive definiteness of $\Sigma$ and $K$.

Now, the inverse linear space $\mathcal{L}^{-1}_G$ can be viewed as the algebraic variety of the kernel of a rational map, 
\[
\rho_G: \mathbb{R}[\Sigma] \rightarrow \mathbb{R}(K), \quad \rho_G(\sigma_{ij})= (i,j)^{th} \text{ entry of } K^{-1}.
\]
The kernel of $\rho_G$ is called the \textit{vanishing ideal} of the model $\mathcal{M}_G$ and is denoted by $I_G$. Note that $I_G$ is a homogeneous prime ideal as it is the kernel of a rational map. The dimension of the ideal is the number of free variables in $\mathbb{R}(K)$, which is $\#V + \#E$. 

\begin{example}
\label{ex:4cyclerhomap}
Let \(G\) be the $4$-cycle as shown in Figure \ref{figure:uncolored 4 cycle}. 
The corresponding concentration matrix of this graph is given by
    \[
    K =  \begin{pmatrix}
    k_{11} & k_{12} & 0 & k_{14}\\
    k_{12} & k_{22} & k_{23} & 0\\
    0 & k_{23} & k_{33} & k_{34}\\
    k_{14} & 0 & k_{34} & k_{44}
    \end{pmatrix}.
    \]
The rational map $\rho_G$ for the $4$-cycle is obtained by inverting $K$ and mapping $\sigma_{ij}$ to the $(i,j)^{th}$ entry of $K^{-1}$. For instance,
\begin{align*}
        \rho_G(\sigma_{11})
        &=  \sigma_{11} 
        =  \big(K^{-1}\big)_{11} 
        = (k_{22} k_{33} k_{44} - k_{22} k_{34}^2 - k_{23}^2 k_{44})/\det(K), \\
        \rho_G(\sigma_{12})
        &= \sigma_{12} 
        = \big(K^{-1}\big)_{12} 
        = (-k_{14} k_{23} k_{34} + k_{12} k_{34}^2 - k_{12} k_{33} k_{44})/\det(K), \text{ and so on.}
    \end{align*}
Computing the kernel of the $\rho_G$ map gives us the vanishing ideal $I_G$, which is generated by
\begin{eqnarray*}
I_G&=&\langle \sigma_{14}\sigma_{23}\sigma_{24}-\sigma_{13}\sigma_{24}^2-\sigma_{14}\sigma_{22}\sigma_{34}+\sigma_{12}\sigma_{24}\sigma_{34}+\sigma_{13}\sigma_{22}\sigma_{44}-\sigma_{12}\sigma_{23}\sigma_{44}, \\
&&\sigma_{13}\sigma_{14}\sigma_{23}- \sigma_{13}^2\sigma_{24}-\sigma_{12}\sigma_{14}\sigma_{33}+\sigma_{11}\sigma_{24}\sigma_{33}+\sigma_{12}\sigma_{13}\sigma_{34}-\sigma_{11}\sigma_{23}\sigma_{34} \rangle.    
\end{eqnarray*}
Note that $I_G$ is a homogeneous prime ideal of dimension $8$. This is because $\rho_G$ is a birational map with $\mathbb{R}(K)$ having $8$ free variables and there is no dependency among them.
\begin{figure}[H]
    \centering
    \begin{tikzpicture}[scale=0.75]
        \node[draw, circle, fill=black] (1) at (0, 2) {};
        \node[draw, circle, fill=black] (2) at (2, 2) {};
        \node[draw, circle, fill=black] (3) at (2, 0) {};
        \node[draw, circle, fill=black] (4) at (0, 0) {};
        
        \draw[thick, black] (1) -- (2);
        \draw[thick, black] (2) -- (3);
        \draw[thick, black] (3) -- (4);
        \draw[thick, black] (4) -- (1);
    
        \node at (-0.5, 2.6) {1};
        \node at (2.5, 2.6) {2};
        \node at (2.5, -0.6) {3};
        \node at (-0.5, -0.6) {4};
    \end{tikzpicture}
    \caption{$4$-cycle \label{figure:uncolored 4 cycle}}
    \end{figure}  
\end{example}


\textbf{Graph coloring and symmetry}: Additional symmetries can be introduced in the model when certain partial correlations interact in the similar way. This was first introduced in \cite{HojsgaardLauritzen2008} where the authors used colored graphs to represent these symmetries. For any given graph $G$, we assign colors to the vertices and edges of $G$ with the condition that the sets of vertex and edge colors are disjoint. Let $\lambda(i)$ and $\lambda(\{i,j\})$ denote the color of the vertex $i$ and edge $\{i,j\}$, respectively, and $\cg$ denote the colored graph. This allows us to introduce the symmetries in the linear space $\mathcal{L}_{\cg}$ in the following way:
\begin{enumerate}
    \item $k_{ij}=0$ if $\{i,j\}$ is not an edge in $\cg$,
    \item $k_{ii}=k_{jj}$ if $\lambda(i)=\lambda(j)$ in $\cg$,
    \item $k_{ij}=k_{xy}$ if $\lambda(\{i,j\})=\lambda(\{x,y\})$ in $\cg$.
\end{enumerate}

Note that introducing these additional symmetries in the model implies that $\mathcal{L}_{\cg}\subset \mathcal{L}_G$, as we reduce the dimension of the linear space by setting certain concentrations equal. This in turn gives us that $\mathcal{L}_{\cg}^{-1}\subset \mathcal{L}_{G}^{-1}$, which also implies that $I_{G}\subset I_{\cg}$ by the ideal-variety correspondence. We illustrate these containment in the following example.

\begin{example}\label{ex:4cycleColoredrhomap}
Let $\cg$ be the colored $4$-cycle as shown in Figure \ref{figure:colored 4 cycle}. The corresponding concentration matrix now has the following form:
    \[
    K =  \begin{pmatrix}
    \textcolor{red}{k_{11}} & \textcolor{blue}{k_{12}} & 0 & \textcolor{green}{k_{23}}\\
    \textcolor{blue}{k_{12}} &  \textcolor{red}{k_{11}} & \textcolor{green}{k_{23}} & 0\\
    0 & \textcolor{green}{k_{23}} & \textcolor{yellow}{k_{33}} & \textcolor{blue}{k_{12}}\\
    \textcolor{green}{k_{23}} & 0 & \textcolor{blue}{k_{12}} & \textcolor{yellow}{k_{33}}
    \end{pmatrix}.
    \]
We consider the variables $k_{11},k_{12},k_{23},$ and $k_{33}$ as the base variables and set $k_{22}=k_{11}, k_{44}=k_{33}, k_{34}=k_{12}$ and $k_{14}=k_{23}$. Computing the vanishing ideal $I_\cg$ gives us
\begin{eqnarray*}
I_\cg&=&\langle \sigma_{33}-\sigma_{44}, 
\sigma_{14}-\sigma_{23}, 
\sigma_{13}-\sigma_{24},
\sigma_{11}-\sigma_{22}, \\
&& \text{ one generator of degree }2, \text{ two generators of degree }3 \rangle,\\
&&\supsetneq I_G.
\end{eqnarray*}
The dimension of \(I_{\cg}\) drops to $4$ since there are only four free variables and the other variables are not independent.
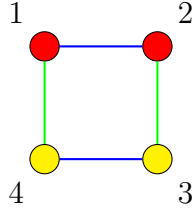
\begin{figure}[H]
    \centering
    \begin{tikzpicture}[scale=0.75]
        \node[draw, circle, fill=red] (1) at (0, 2) {};
        \node[draw, circle, fill=red] (2) at (2, 2) {};
        \node[draw, circle, fill=yellow] (3) at (2, 0) {};
        \node[draw, circle, fill=yellow] (4) at (0, 0) {};
        
        \draw[thick, blue] (1) -- (2);
        \draw[thick, green] (2) -- (3);
        \draw[thick, blue] (3) -- (4);
        \draw[thick, green] (4) -- (1);
    
        \node at (-0.5, 2.6) {1};
        \node at (2.5, 2.6) {2};
        \node at (2.5, -0.6) {3};
        \node at (-0.5, -0.6) {4};
    \end{tikzpicture}
    \caption{Colored $4$-cycle \label{figure:colored 4 cycle}}
\end{figure}  
\end{example}
Observe that in the previous example, we obtained some linear binomials in $I_{\cg}$ which were not present in $I_G$. In general, the only time we get linear relations in $I_G$ is when $G$ is a disconnected graph, and the linear relations are of the form $\sigma_{ij}$, where $i$ and $j$ are disconnected in $G$ (Proposition 2.5 \cite{DaviesMarigliano:2021}). This follows from the fact that all the $k_{ij}$s are independent variables in uncolored graphs. However, linear polynomials can show up in $I_\cg$ due to the dependency between concentrations. This observation was first recorded in \cite{SturmfelsUhler2010}. This study was further continued by Marigliano and Davies in \cite{DaviesMarigliano:2021}, where they analyzed the linear relations obtained in $I_{\cg}$ for colored cycles. Specifically, they made a connection between the linear binomials obtained in $I_\cg$ and the graph symmetries of the colored cycles. We thus include the definition and some properties of graph symmetry below for the sake of completion.

\begin{definition} 
\label{def:graphsym}
For a given colored graph $\cg$ and its concentration matrix $K$, a symmetry of $\cg$ is a permutation matrix $P$ such that $PKP^{-1}=K$.
\end{definition}

It is easy to see that such symmetries give rise to linear binomials which lie in $I_\cg$.

\begin{proposition}
\cite[Prop. 2.2]{DaviesMarigliano:2021}
\label{prop:symmetries}
Let $\cg$ be a colored graph and $P$ be a symmetry of $\cg$. If $\Sigma$ is a generic covariance matrix, then the linear binomials defined by all distinct entries of $P \Sigma P^{-1} - \Sigma$ belong to $I_\cg$.
\end{proposition}

Now, a colored graph $\cg$ is said to have a uniform coloring if every vertex has the same color and every edge has the same color. In \cite{DaviesMarigliano:2021}, the authors focused on specific graphs with uniform coloring and identified all the linear binomials in $I_\cg$. Specifically, they proved the following theorem:

\begin{theorem}
\cite[Theorem 3.4]{DaviesMarigliano:2021}
\label{Theo:UniformColoring}
Let $C_n$ be the $n$-cycle of uniform coloring. Then the linear part of $I_{C_n}$ is induced by symmetries and consists of the relations: 
        \[
        \sigma_{11+d} - \sigma_{ii+d}, \qquad \text{for } i \in \{2,...,n\} \text{ and } d \in \{0,..., \left\lfloor \frac{n}{2} \right\rfloor \},
        \]
        where all indices are taken modulo n.
\end{theorem}  

Based on the findings from the above theorem, the authors conjectured that the result would also hold for any colored cycle.

\begin{conjecture}[Conjecture 4.2 \cite{DaviesMarigliano:2021}]\label{Conj:ifandonlyif}
Let $\cg$ be a colored $n$-cycle. Then a linear binomial lies in $I_\cg$ if and only if there is a corresponding symmetry in $\cg$.
\end{conjecture}

Now, the sufficient condition of the conjecture follows from Proposition \ref{prop:symmetries}. Thus, we focus on the necessary condition of the conjecture and analyze if and when the condition is true. It is a well known fact that a permutation matrix $P$ satisfies $P A P^{-1} = A$ where $A$ is the adjacency matrix of a graph $G$ with $n$ vertices, if and only if $P$ corresponds to an element in $D_n$. As the concentration matrix $K$ can be seen as a weighted adjacency matrix, every graph symmetry of a colored $n$-cycle corresponds to a rotation or a reflection. We explain this connection of graph symmetry with reflection and rotation of an $n$-cycle and with the linear binomials in the example below.

\begin{example}
Let \(\cg\) be the colored $4$-cycle as shown in Figure \ref{figure:colored 4 cycle}. As seen Example~\ref{ex:4cycleColoredrhomap}, the vanishing ideal $I_\cg$ has four linear binomials:
\[
    \sigma_{33} - \sigma_{44}, 
    \sigma_{14} - \sigma_{23}, 
    \sigma_{13} - \sigma_{24},  \text{ and }
    \sigma_{11} - \sigma_{22}.
    \]
Now observe that $\cg$ has a graph symmetry, which is the reflection through the axis passing through edges \(\{1,2\}\) and \(\{3,4\}\). Thus, the corresponding permutation matrix is the one, which interchanges vertex \(\{1\}\) with \( \{2\}\) and vertex \(\{3\}\) with \(\{4\}\). Specifically, we get the following permutation matrix $P$ that satisfies $PKP^{-1}=K$:
\[
    P = \begin{pmatrix}
        0 & 1 & 0 & 0 \\
        1 & 0 & 0 & 0\\
        0 & 0 & 0 & 1\\
        0 & 0 &  1 & 0
        \end{pmatrix}.
\]
It can be easily checked that the four linear binomials are obtained from the equation $P\Sigma P^{-1}-\Sigma$ in this example.
\end{example}

Apart from the algebraic motivation to explore the connection between the linear binomials and symmetries of a colored cycle, there also lies a computational advantage in learning these linear binomials. Recall that in order to compute the vanishing ideal, we need to compute the kernel of the polynomial map $\rho_\cg$ which is described above. The domain of $\rho_\cg$ is all the covariances $\sigma_{ij}$ with $i\leq j \leq n$. Thus, computing the vanishing ideal can be time consuming, especially for larger cycles. However, if we already have the information about the linear binomials from the structure of the graph (either from symmetry or some other property), then we can reduce the size of the domain by removing certain variables, making the computation comparatively faster. The effect of this strategy is demonstrated in the example below.

\begin{example}
Let $\cg$ be the colored $6$-cycle as shown in Figure \ref{figure:colored 6 cycle}.
The vanishing ideal of this graph contains the following binomial linear forms:
\[
    \sigma_{22} - \sigma_{66}, \ \
    \sigma_{33} - \sigma_{55}, \ \
    \sigma_{12} - \sigma_{16}, \ \
    \sigma_{23} - \sigma_{56}, \ \
    \sigma_{34} - \sigma_{45}, \ \
    \sigma_{24} - \sigma_{46}, \ \
    \sigma_{13} - \sigma_{15}, \ \
    \sigma_{25} - \sigma_{36}.
\]
Computing the full ideal (without taking symmetries into account) takes approximately 12.0517 seconds.
However, observing that the graph has a reflection symmetry swapping vertices \(2\) with \(6\), and vertex \(3\) with \(5\), we can predict these binomial linear generators.
By removing the variables \(\sigma_{22}, \sigma_{33}, \sigma_{12}, \sigma_{23}, \sigma_{34}, \sigma_{24}, \sigma_{13}\) and \(\sigma_{25}\) from the computation, the size of the domain is reduced. As a result, computing the remaining ideal now only takes 6.32652 seconds, a reduction of nearly 47.5\% in computational time. Even when accounting for the time required to compute the graph symmetries, which is approximately 2.35788 seconds, the overall process remains faster than the full computation, with a net gain of around 27.9\% in efficiency.
\begin{figure}[H]
    \centering
    \begin{tikzpicture}[scale=0.65, every node/.style={circle, draw, inner sep=1pt, minimum size=4mm, font=\small}]
        \node[fill=red, label=above:{1}] (1) at (90:2) {};
        \node[fill=yellow, label=right:{2}] (2) at (30:2) {};
        \node[fill=yellow, label=below:{3}] (3) at (330:2) {};
        \node[fill=purple, label=below:{4}] (4) at (270:2) {};
        \node[fill=yellow, label=left:{5}] (5) at (210:2) {};
        \node[fill=yellow, label=above left:{6}] (6) at (150:2) {};
    
        \draw[blue, thick] (1) -- (2);
        \draw[green, thick] (2) -- (3);
        \draw[green, thick] (3) -- (4);
        \draw[green, thick] (4) -- (5);
        \draw[green, thick] (5) -- (6);
        \draw[blue, thick] (6) -- (1);
    \end{tikzpicture}
    \caption{Colored $6$-cycle}\label{figure:colored 6 cycle}
\end{figure}
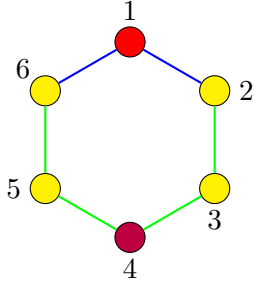 
\end{example}

In order to explore the necessary condition of the conjecture, it is important to first understand the image of each $\sigma_{ij}$ under the $\rho_\cg$ map. In \cite{JonesWest:2005}, the authors developed a combinatorial connection between the image of $\sigma_{ij}$ and the structure of the graph. 

\begin{theorem}
\label{theorem:covariance}\cite[Theorem 1]{JonesWest:2005}
Consider an $n$-dimensional multivariate normal distribution with a finite and non-singular covariance matrix $\Sigma$, and concentration matrix $K = \Sigma^{-1}$. The element of $\Sigma$ corresponding to the covariance between vertices $i$ and $j$ can be written as a sum of path weights over all paths in the graph between $i$ and $j$:
\begin{align*}
    \sigma_{ij} = \sum_{P \in \mathcal{P}_{ij}} (-1)^{m+1} k_{p_1 p_2} k_{p_2 p_3} \dots k_{p_{m-1} p_m} \frac{\det(K_{\setminus P})}{\det(K)},
\end{align*}
where $\mathcal{P}_{ij}$ represents the set of paths between $i$ and $j$, such that $p_1 = i$ and $p_m = j$ for all $P \in \mathcal{P}_{ij}$, and $K_{\setminus P}$ is the matrix with rows and columns corresponding to the variables in the path $P$ omitted, with the determinant of a zero-dimensional matrix taken to be 1.
\end{theorem}

Implementing the above result on cycles provides us a simplified way to analyze the image of $\sigma_{ij}$. This is because for any two distinct vertices $i$ and $j$ in a cycle, there exists exactly two paths between them. Thus, the next section is focused on analyzing the paths and the determinant of their corresponding concentration matrices.

\section{Analysis of path graphs}\label{section:analysis of path graphs}
As mentioned in the end of Section \ref{Section:Preliminaries}, implementing Theorem \ref{theorem:covariance} on cycles allows us to view the image of each $\sigma_{ij}$ as the sum over the two paths between $i$ and $j$ in the cycle. We state this result as an immediate corollary to Theorem \ref{theorem:covariance}.

\begin{corollary}\label{Cor:J&WforCycle}
Let $G$ be an $n$-cycle with concentration matrix $K$ and let $i$ and $j$ be vertices in $G$. Then, there exist precisely two distinct paths connecting \( i \) and \( j \) in $G$. Let the shorter path be denoted by \( i \leftrightarrow j \), and the complementary path by \( i \overset{c}{\leftrightarrow} j \). Then covariance between the vertices $i$ and $j$ is given by
\[
        \sigma_{ij} = \frac{1}{\operatorname{det}(K)} 
        \big( (-1)^{n_{i \leftrightarrow j}\;+\;1} \prod_{\{i^{\prime}, j^{\prime}\} \in i \leftrightarrow j} k_{i^{\prime} j^{\prime}} \operatorname{det}\big(K_{\setminus i \leftrightarrow j}\big)
        + (-1)^{n_{i \overset{c}{\leftrightarrow} j}\;+\;1} \prod_{\{i^{\prime}, j^{\prime}\} \in i \overset{c}{\leftrightarrow} j} k_{i^{\prime} j^{\prime}} \operatorname{det}\big(K_{\setminus i \overset{c}{\leftrightarrow} j}\big)\big),
\]
where $n_{i \leftrightarrow j}$ denotes the number of vertices on the path $i \leftrightarrow j$, and $K_{\setminus i \leftrightarrow j}$ is the submatrix of the concentration matrix $K$ obtained after removing the rows and columns corresponding to the vertices in $i \leftrightarrow j$.
\end{corollary}

\begin{example}\label{example:sum of two paths}
Let \(\cg\) be the colored $4$-cycle as shown in Example~\ref{ex:4cycleColoredrhomap}.
Applying Corollary~\ref{Cor:J&WforCycle}, we get the image of $\sigma_{14}$ and $\sigma_{23}$ as follows:
\begin{align*}
&\sigma_{14} = \frac{1}{\operatorname{det}(K)} \big( (-1)^{2+1} k_{14} \operatorname{det}\big(K_{\setminus 1 \leftrightarrow 4}\big)
+  (-1)^{4+1} k_{12} k_{23} k_{34} \operatorname{det}\big(K_{\setminus 1 \overset{c}{\leftrightarrow} 4}\big) \big),\\
&\sigma_{23} = \frac{1}{\operatorname{det}(K)}
\big( (-1)^{2+1} k_{23} \operatorname{det}\big(K_{\setminus 2 \leftrightarrow 3}\big)
+  (-1)^{4+1} k_{34} k_{14} k_{12} \operatorname{det}\big(K_{\setminus 2 \overset{c}{\leftrightarrow} 3}\big) \big).
\end{align*}
The determinants are simplified as
\begin{align*}
&\operatorname{det}\big(K_{\setminus 1 \leftrightarrow 4}\big) = \operatorname{det}
\begin{pmatrix}
    k_{22} &k_{23}\\
    k_{23} & k_{33}
\end{pmatrix}, \quad
\operatorname{det}\big(K_{\setminus 1 \overset{c}{\leftrightarrow} 4} \big)= 1,\\
&\operatorname{det}\big(K_{\setminus 2 \leftrightarrow 3}\big) = \operatorname{det}
\begin{pmatrix}
    k_{11} &k_{14}\\
    k_{14} & k_{44}
\end{pmatrix}, \quad
\operatorname{det}\big(K_{\setminus 2 \overset{c}{\leftrightarrow} 3}\big) = 1.
\end{align*}
We adopt the convention from Theorem~\ref{theorem:covariance} that the determinant of a zero-dimensional matrix is $1$. Substituting the color constraints \(k_{12} = k_{34}\), \( k_{23} = k_{14} \), \(k_{11} = k_{22} \), and \(k_{33} = k_{44} \) gives us that 
\[
k_{14}\det(K_{\setminus 1\leftrightarrow 4}) = k_{23}\det(K_{\setminus 2\leftrightarrow 3})  \text{ and } k_{12}k_{23}k_{34}\det(K_{\setminus 1\overset{c}\leftrightarrow 4}) = k_{34}k_{14}k_{12}\det(K_{\setminus 2\overset{c}\leftrightarrow 3}),
\]
implying that $\sigma_{14}-\sigma_{23} \in I_\cg$.
\end{example}

Using the path expression of $\sigma_{ij}$ in Corollary \ref{Cor:J&WforCycle}, one can obtain a sufficient condition for the existence of a linear binomial in $I_\cg$. In particular, $\sigma_{ij}-\sigma_{xy} \in I_\cg$ for a given colored cycle $\cg$ if the following conditions hold:
\begin{enumerate}
\item $\prod_{\{i',j'\}\in i\leftrightarrow j} k_{i'j'}=\prod_{\{i',j'\}\in x\leftrightarrow y} k_{i'j'}$,
\item $\prod_{\{i',j'\}\in i \overset{c}\leftrightarrow j} k_{i'j'}=\prod_{\{i',j'\}\in x \overset{c}\leftrightarrow y} k_{i'j'}$,
\item $\det(K_{\setminus i \leftrightarrow j})=\det(K_{\setminus x \leftrightarrow y})$, and
\item $\det(K_{\setminus i \overset{c}\leftrightarrow j})=\det(K_{\setminus x \overset{c}\leftrightarrow y})$.
\end{enumerate}
In other words, $\sigma_{ij}-\sigma_{xy}$ lies in the vanishing ideal of $\cg$ if the length and the edge color multiplicities of the paths $i \leftrightarrow j$ and $x \leftrightarrow y$ and similarly $i \overset{c}\leftrightarrow j$ and $x \overset{c}\leftrightarrow y$ are equal, along with  $\det(K_{\setminus i \leftrightarrow j})=\det(K_{\setminus x \leftrightarrow y})$  and $\det(K_{\setminus i \overset{c}\leftrightarrow j})=\det(K_{\setminus x \overset{c}\leftrightarrow y})$. This condition can also be seen in Example \ref{example:sum of two paths}. Although this is not a necessary condition as there can be cancellations among the terms of the two paths of $\sigma_{ij}$ (which we discuss in a later section), we focus on when this sufficient condition holds in this section. In particular, our goal now is to study the conditions when $\det(K_{\setminus i \leftrightarrow j})$ and $\det(K_{\setminus x \leftrightarrow y})$ (and similarly the complementary paths) are equal.

\subsection{Conditions on two colored path graphs having the same determinant} 
In order to analyze the determinant of the concentration matrix of a colored path, we first obtain a way to express the determinant in terms of the edges and vertices of the path. We explore the fact that the concentration matrix of a path is a tridiagonal matrix to obtain the following result:

\begin{lemma}\label{lem:tridiagonal det}
Let $P$ be a path graph on $m$ vertices and $K_{P} \in \mathbb{R}^{m \times m}$ be the corresponding concentration matrix. Then $K_{P}$ is a tridiagonal matrix and the determinant is given by: 
    \begin{align*}
    \operatorname{det}(K_{P}) = 
    \sum_{\substack{|S| = 0, \\ S \subseteq E_P \text{ disjoint}}}^{\left\lfloor \frac{m}{2} \right\rfloor} (-1)^{|S|} \prod_{\{i,j\} \in S} k_{ij}^2 \prod_{v \in V_P \setminus V(S)} k_{vv},
    \end{align*}
where $E_P$ denotes the edge set of $P$, $V(S)$ the set of vertices incident to the edges in $S$, and $|S|$ the number of edges in the $S$.
\end{lemma}

\begin{proof}
The concentration matrix $K_P$ is a tridiagonal matrix as every vertex \( i \in \{2, \dots, m-1\} \) in $P$ is connected to exactly two neighbors \( i-1 \) and \( i+1 \), whereas the vertices \( 1 \) and \( m \) are only connected to \( 2 \) and \( m-1 \), respectively.
Let \( K_m\) denote the concentration matrix of the path of length $m$, $E_{m}$ denote the $m-1$ edges, and $V_m$ denote the corresponding set of $m$ vertices. We apply induction on the length of the path to obtain the formula. 
    
The recurrence relation for the determinant of a tridiagonal matrix is given by:
\[
\det (K_m)=k_{mm} \det(K_{m-1}) -k_{m-1m}^2 \det(K_{m-2}).
\]
Since $m = 1$ is the trivial case, we start with $m = 2$ as the base case, which is the path with two vertices. Using the above recursion, the determinant of $K_2$ can be written as
\[
\det(K_2)=k_{11} k_{22} - k_{12}^2 = (-1)^{|\emptyset|} \prod_{v\in [2]} k_{vv} + (-1)^{|\{1,2\}|} \prod_{\{i,j\}\in \{1,2\}}k_{ij}^2.
\]
Assuming that the formula holds for all subpaths of $P$ of length smaller or equal to \( m-1 \), we expand the determinant of $K_m$. Applying the recurrence relation and our induction hypothesis, we get
\begin{eqnarray*}
 \det(K_m) &=& k_{mm} \det(K_{m-1}) - k_{m-1m}^2 \det(K_{m-2})   \\
 &=& k_{mm} \sum_{\substack{|S| = 0, \\ S \subseteq E_{m-1} \text{ disjoint}}}^{\left\lfloor \frac{m-1}{2} \right\rfloor} (-1)^{|S|} \prod_{\{i,j\} \in S} k_{ij}^2 \prod_{v \in V_{m-1} \setminus V(S)} k_{vv} 
        \ - \\
        &&\ k_{m-1m}^2 \sum_{\substack{|S| = 0, \\ S \subseteq E_{m-2} \text{ disjoint}}}^{\left\lfloor \frac{m-2}{2} \right\rfloor} (-1)^{|S|} \prod_{\{i,j\} \in S} k_{ij}^2 \prod_{v \in V_{m-2} \setminus V(S)} k_{vv}.
\end{eqnarray*}
For $S\subseteq E_{m-1}$, the vertex $m$ is never included in $v \in V_{m-1} \setminus V(S)$. Hence, multiplying the product $k_{vv}$ over these vertices by $k_{mm}$ is equivalent to taking the product of $k_{vv}$ over $v \in V_{m} \setminus V(S)$.
Moreover, multiplying the second sum by \( k_{m-1m}^2 \) corresponds to summing over all edge sets \( S \subseteq E_{m-2} \) that additionally include the edge \( \{m-1,m\} \). 
Since \( S \) then contains one more edge than in the previous sum, the upper bound of the sum increases by 1 and the sign switches. This new $S$, now denoted by $S'$, ensures that the vertices \( m-1 \) and \( m \) are always included in \( V(S') \). Consequently, we can rewrite the equation as:
\begin{eqnarray*}        
 \det(K_m)&=& \sum_{\substack{|S| = 0, \\ S \subseteq E_{m-1} \text{ disjoint}}}^{\left\lfloor \frac{m-1}{2} \right\rfloor} (-1)^{|S|} \prod_{\{i,j\} \in S} k_{ij}^2 \prod_{v \in V_{m} \setminus V(S)} k_{vv} + \\
 &&\sum_{\substack{|S'| = 1, \\ S'=S\cup \{m-1,m\} \\ S \subseteq E_{m-2} \text{ disjoint} }}^{\left\lfloor \frac{m}{2} \right\rfloor} (-1)^{|S|} \prod_{\{i,j\} \in S'} k_{ij}^2 \prod_{v \in V_{m} \setminus V(S')} k_{vv}.
\end{eqnarray*}

In the first sum, since $E_{m-1} = E_m \setminus \{m-1,m\}$, we can rewrite $S \subseteq E_{m-1}$ disjoint as $S \subseteq E_{m}$ disjoint with the additional condition that $\{m-1,m\} \notin S$. Furthermore, the upper bound of that sum can be increased by 1 since, for $|S| = \left\lfloor \frac{m}{2} \right\rfloor$, the sum is empty as the edges in $S$ need to be disjoint, and $\{m-1,m\} \notin S$. In the second sum, we have $\{m-1,m\} \in S'$, and we also avoid the edge $\{m-2, m-1\}$ in all summand as $S\in E_{m-2}$. Thus, we can rewrite $S \subseteq E_{m-2}$ disjoint with $S'=\{m-1,m\}\cup S$ as $S \subseteq E_{m}$ disjoint with $\{m-1,m\} \in S$. The above sum can then be rewritten as:
\begin{eqnarray*}
\det(K_m)&=& \sum_{\substack{|S| = 0, \\ S \subseteq E_{m} \text{ disjoint}, \\ \{m-1,m\} \notin S}}^{\left\lfloor \frac{m}{2} \right\rfloor} (-1)^{|S|} \prod_{\{i,j\} \in S} k_{ij}^2 \prod_{v \in V_{m} \setminus V(S)} k_{vv} + \\
&&\sum_{\substack{|S| = 1, \\  S \subseteq E_{m} \text{ disjoint}, \\ \{m-1,m\} \in S}}^{\left\lfloor \frac{m}{2} \right\rfloor} (-1)^{|S|} \prod_{\{i,j\} \in S} k_{ij}^2 \prod_{v \in V_{m} \setminus V(S)} k_{vv}, \text{ which is equal to}\\ 
&&=\sum_{\substack{|S| = 0, \\ S \subseteq E_m \text{ disjoint}}}^{\left\lfloor \frac{m}{2} \right\rfloor} (-1)^{|S|} \prod_{\{i,j\} \in S} k_{ij}^2 \prod_{v \in V_m \setminus V(S)} k_{vv}.
\end{eqnarray*}
\end{proof}

Now, for any given colored path $P$, the obvious two candidates for a colored path $Q$ to satisfy $\det(K_P)=\det(K_Q)$ are when $Q$ is identical to $P$ or $Q$ is the reflection of $P$. This follows from the fact that if $Q$ is identical or the reflection of $P$, then $K_P$ and $K_Q$ are similar matrices with $K_Q=P K_P P^{-1}$, where $P$ is the identity matrix (when $Q$ is identical) or the anti-diagonal matrix (when $Q$ is the reflection). In the next lemma, we show that these are indeed the only two candidates when $P$ is a colored path of length one or two. Using this result, we also show that the conjecture is true for $3,5,$ and $7$ cycles.

\begin{lemma}\label{Lemma:Pathlength 1 and 2}
Let $P$ and $Q$ be two colored paths, both of length one or two. If $\det(K_P) = \det(K_Q)$, then $P$ and $Q$ are either identical or reflections of each other.
\end{lemma}
\begin{proof}
We analyze the two possible path lengths separately:
    
\textbf{1)} Let $P$ and $Q$ be paths of length one with $\operatorname{det}(K_P) = \operatorname{det}(K_Q)$. Furthermore, let the vertices of $P$ and $Q$ be $v, w$ and $x,y$, respectively.
        The associated concentration matrices are given by:
    \[
        K_P =
        \begin{pmatrix}
            k_{vv} & k_{vw} \\
            k_{vw} & k_{ww}
        \end{pmatrix},
        \hspace{1em}
        K_Q =
        \begin{pmatrix}
            k_{xx} & k_{xy} \\
            k_{xy} & k_{yy}
        \end{pmatrix}.
    \]
By equating the determinants, we have: 
    \[
     k_{vv} k_{ww} - k_{vw}^2 = k_{xx} k_{yy} - k_{xy}^2.
    \]
As the vertex and edges colors are disjoint, two monomials can be equal only if they have the same vertex and edge degree. Thus, we have $k_{vw} = k_{xy}$. This gives rise to two cases:
    \begin{enumerate}
        \item If $k_{vv} = k_{xx}$ and $k_{ww} = k_{yy}$, then $P$ and $Q$ are identical.
        \item  If $k_{vv} = k_{yy}$ and $k_{ww} = k_{xx}$, the $P$ and $Q$ are reflections of each other.
    \end{enumerate}
    
\textbf{2)} Now, let \(P\) and \(Q\) be paths of length two, with vertices of \(P\) and $Q$ as \(u, v, w\), and \(x, y, z\), respectively. The associated concentration matrices are given by:
    \[
        K_P =
        \begin{pmatrix}
            k_{uu} & k_{uv} & 0 \\
            k_{uv} & k_{vv} & k_{vw}\\
            0 & k_{vw} & k_{ww}
        \end{pmatrix},
        \hspace{1em}
        K_Q =
        \begin{pmatrix}
            k_{xx} & k_{xy} & 0 \\
            k_{xy} & k_{yy} & k_{yz}\\
            0 & k_{yz} & k_{zz}
        \end{pmatrix}.
    \]
Setting the two determinants equal gives us:
    \[
         k_{uu} k_{vv} k_{ww} - k_{uv}^2 k_{ww} - k_{vw}^2 k_{uu} 
         = k_{xx} k_{yy} k_{zz} - k_{xy}^2 k_{zz} - k_{yz}^2 k_{xx}.
    \]
Since the vertex and edge degrees need to match, we know that $k_{uu} k_{vv} k_{ww} = k_{xx} k_{yy} k_{zz}$.  This results in two possible cases:
    \begin{enumerate}
        \item If $k_{uv}^2 k_{ww} = k_{xy}^2 k_{zz}$ and $k_{vw}^2 k_{uu} = k_{yz}^2 k_{xx}$, it follows that $k_{uv} = k_{xy}$ and $k_{ww} = k_{zz}$. Additionally, $k_{vw} = k_{yz}$ and $k_{uu} = k_{xx}$. Since $k_{uu}k_{vv}k_{ww}=k_{xx}k_{yy}k_{zz}$ we have $k_{vv} = k_{yy}$, and therefore $P$ and $Q$ are identical.
        \item If $k_{uv}^2 k_{ww} = k_{yz}^2 k_{xx}$ and $k_{vw}^2 k_{uu} = k_{xy}^2 k_{zz}$, we get that $k_{uv}= k_{yz}$, $ k_{ww} = k_{xx}$ and $k_{vw} = k_{xy}$, $ k_{uu} = k_{zz}$, which implies $k_{vv} = k_{yy}$.
        In this case, the paths $P$ and $Q$ are reflections of each other.
    \qedhere
    \end{enumerate}
\end{proof}

\begin{theorem}\label{theorem:3,5,7 cycle}
Let $\cg$ be a colored $3,5$ or $7$ cycle. Then a linear binomial lies in $I_\cg$ if and only if there is a corresponding symmetry in $\cg$.
\end{theorem}

\begin{proof}
We first prove the theorem for $5$ cycles. In a 5-cycle, the shorter path between any two distinct vertices is either of length one (with the complementary path of length four), or of length two (with the complementary path of length three). This configuration gives us three possible types of linear binomials that can arise in $I_\cg$:
    
\textbf{Case I: Shorter path between $i$ and $j$ is of length one:} Without loss of generality we assume $\sigma_{12} - \sigma_{45} \in I_\cg$. Then, by Corollary~\ref{Cor:J&WforCycle}, we have:
\begin{eqnarray*}
\sigma_{12}&=& k_{12}\det(K_{3\leftrightarrow 5}) - k_{23}k_{34}k_{45}k_{15}\det(K_{\emptyset}), \text{ and} \\
\sigma_{45}&=& k_{45}\det(K_{1\leftrightarrow 3}) - k_{34}k_{23}k_{12}k_{15}\det(K_{\emptyset}).
\end{eqnarray*}
Notice that in both $\sigma_{12}$ and $\sigma_{45}$, the second sum is a monomial with vertex degree $0$ and edge degree $4$. As there is no term in the first sum which can have a similar degree, we can conclude that
\begin{eqnarray*}
k_{12}\det(K_{3\leftrightarrow 5}) &=& k_{45}\det(K_{1\leftrightarrow 3}), \text{ and} \\ k_{23}k_{34}k_{45}k_{15}\det(K_{\emptyset})&=& k_{34}k_{23}k_{12}k_{15}\det(K_{\emptyset}).
\end{eqnarray*}
Therefore, $k_{12} = k_{45}$. Since the paths on $V(3 \leftrightarrow 5)$ and $V(1 \leftrightarrow 3)$ are both of length two, by Lemma~\ref{Lemma:Pathlength 1 and 2}, they are either identical or reflections of each other.
    \begin{enumerate}
        \item If the paths on $V(3 \leftrightarrow 5)$ and $V(1 \leftrightarrow 3)$ are equal to each other, it follows that $k_{33} = k_{11}$, $k_{44} = k_{22}$, $k_{55} = k_{33}$, and $k_{34} = k_{12}$, $k_{45} = k_{23}$. Thus, there is a reflection symmetry in the graph mapping vertex 1 to vertex 5 and vertex 2 to vertex 4.
        \item If the paths $V(3 \leftrightarrow 5)$ and $V(1 \leftrightarrow 3)$ are reflections of each other, it holds that $k_{44} = k_{22}$, $k_{55} = k_{11}$, and $k_{34} = k_{23}$. In this case, there exists a reflection symmetry mapping vertex 1 to vertex 5 and vertex 2 to vertex 4.
    \end{enumerate}

\textbf{Case II: Shorter path between $i$ and $j$ is of length two:} Again without loss of generality, we assume that $\sigma_{13} - \sigma_{24} \in I_\cg$. Using a similar degree argument as in the previous case, we can conclude that 
\begin{eqnarray*}
   k_{12}k_{23}\det(K_{4\leftrightarrow 5}) &=& k_{23}k_{34}\det(K_{5\leftrightarrow 1}), \text{ and} \\
   k_{34}k_{45}k_{15} \operatorname{det}(K_{2\leftrightarrow 2}) &=& k_{45}k_{15}k_{12} \operatorname{det}(K_{3\leftrightarrow 3}).
\end{eqnarray*}
Thus, $k_{12} = k_{34}$ and $k_{22} = k_{33}$. Since the paths on $V(4 \leftrightarrow 5)$ and $V(5 \leftrightarrow 1)$ are both of length one, we can apply Lemma~\ref{Lemma:Pathlength 1 and 2} to obtain the following two cases:
    \begin{enumerate}
        \item If the paths on $V(4 \leftrightarrow 5)$ and $V(5 \leftrightarrow 1)$ are equal, then it follows $k_{44} = k_{55} = k_{11}$ and $k_{45} = k_{15}$. This implies that there exists a reflection symmetry in the cycle mapping vertex $1$ to vertex $4$ and vertex $2$ to vertex $3$.
        \item If the paths on $V(4 \leftrightarrow 5)$ and $V(5 \leftrightarrow 1)$ are reflections of each other, it holds that $k_{44} = k_{11}$ and $k_{45} = k_{15}$. Thus, there exists a reflection symmetry mapping vertex 1 to vertex 4 and vertex 2 to vertex 3.
    \end{enumerate}
    
\textbf{Case III: $i$ and $j$ are equal:} In this case, we get a linear binomial of the form  \( \sigma_{ii} - \sigma_{xx} \). Without loss of generality, we can assume that \( \sigma_{11} - \sigma_{33} \in I_\cg \). By Corollary~\ref{Cor:J&WforCycle}, we know that $\det(K_{2\overset{c}\leftrightarrow 5})=\det(K_{4\overset{c}\leftrightarrow 2})$, which implies
\begin{eqnarray*}
&k_{22} k_{33} k_{44} k_{55} - k_{23}^2 k_{44} k_{55} - k_{34}^2 k_{22} k_{55} - k_{45}^2 k_{22} k_{33} + k_{23}^2 k_{45}^2 =\\
&k_{44} k_{55} k_{11} k_{22} - k_{45}^2 k_{11} k_{22} - k_{15}^2 k_{44} k_{22} - k_{12}^2 k_{44} k_{55} + k_{45}^2 k_{12}^2.
\end{eqnarray*}        
Comparing the monomials according to their vertex and edge degrees gives us that $k_{22} k_{33} k_{44} k_{55}=k_{44} k_{55} k_{11} k_{22}$, and $k_{23}^2 k_{45}^2=k_{45}^2 k_{12}^2$. Thus, we get $k_{33} = k_{11}$ and $k_{23} = k_{12}$. Substituting these equalities gives us an additional equality, which is $k_{34}^2 k_{22} k_{55} = k_{15}^2 k_{44} k_{22}$. This implies that $k_{34} = k_{15}$ and $k_{44} = k_{55}$. Therefore, there exists a reflection symmetry mapping vertex $1$ to vertex $3$ and vertex $4$ to vertex $5$.

Thus, we have shown that for any possible linear binomial that can appear in $I_\cg$, there must exist a corresponding reflection symmetry in $\cg$. A similar argument follows for colored cycles of length $3$. 

We now prove the theorem for $n=7$. Observe that the length of the shorter path between any two distinct vertices $i$ and $j$ can vary from one to three. However, the length of the complementary paths can be four or five. So, we first analyze a binomial where the complementary path is of length five. Without loss of generality, we assume that \(\sigma_{12} - \sigma_{67} \in I_{\cg}\), where \(\cg\) is a colored $7$-cycle.
By Corollary~\ref{Cor:J&WforCycle}, we know that 
\[
k_{12} \operatorname{det}(K_{3 \overset{c}{\leftrightarrow}7}) = 
k_{67} \operatorname{det}(K_{1 \overset{c}{\leftrightarrow}5}).
\]
Hence, it holds \(k_{12} = k_{67} \).
Expanding the equation \( \operatorname{det}(K_{3 \overset{c}{\leftrightarrow}7}) =  \operatorname{det}(K_{1 \overset{c}{\leftrightarrow}5}) \) and comparing the monomials with the same vertex and edge degree gives us
that $k_{66} k_{77} = k_{11} k_{22}$, and 
$k_{34}^2 k_{56}^2 k_{77} = k_{23}^2 k_{45}^2 k_{11}$, implying that $ k_{34} k_{56} = k_{23} k_{45}$ and $k_{77} = k_{11}$ and thus, $k_{66} = k_{22}$.
Similarly, matching the monomials with edge degree one, we get $k_{56}^2 k_{33} k_{44} k_{77} = k_{23}^2 k_{11} k_{44} k_{55}$, so $k_{56} = k_{23}$ and $k_{33} = k_{55}$. Therefore, $k_{34} = k_{45}$ and there exists a symmetry axis in the graph mapping vertex $1$ to vertex $7$, vertex $2$ to vertex $6$ and vertex $3$ to vertex $5$. A similar argument can be repeated when the complementary path is of length four.
\end{proof}

Note that a similar argument does not work for colored $4$-cycles as the two paths between $i$ and $j$ can have the same length, which prevents us from using the same vertex and edge degree argument. Similarly, for cycles of length $9$ and above, the length of the shorter paths can be larger than $3$, and hence Lemma \ref{Lemma:Pathlength 1 and 2} is no longer applicable.
Now, having shown that for paths of size $1$ or $2$, $P$ and $Q$ need to be identical or reflection of each other in order to have $\det(K_P)=\det(K_Q)$, we show next that this is not the case for longer paths. In particular, we come up with a color configuration, each for odd and even length paths such that $P$ and $Q$ are neither identical nor reflection of each other but still satisfy $\det(K_P)=\det(K_Q)$.

\begin{theorem}
\label{theorem:pathDetTypeEven}
Let $P$ and $Q$ be two colored paths on $m$ vertices with vertex sets $V_P = \{p_1, p_2, \dots, p_m\}$ and $V_Q = \{q_1, q_2, \dots, q_m\}$, and edge sets \(E_P = \{ \{p_i,p_{i+1}\} \mid i,j \in \{1,2, \dots, m-1\} \}\) and \(E_Q = \{ \{q_i,q_{i+1}\} \mid i,j \in \{1,2, \dots, m-1\} \}\), respectively. Let $m$ be even and the coloring of $P$ and $Q$ satisfy the following conditions:
\begin{enumerate}
\item $\lambda(\{p_i,p_{i+1}\}) = \lambda(\{q_i,q_{i+1}\})$ for every $i\in \{1,2,\ldots, m-1\}$,
\item $\lambda(p_1)=\lambda(p_{2n+1})$ and $\lambda(q_1)=\lambda(q_{2n+1})$ for every $n\in \{1,2,\ldots, m/2-1\}$ (odd vertices have the same color),
\item $\lambda(p_2)=\lambda(p_{2n})$ and $\lambda(q_2)=\lambda(q_{2n})$ for every $n\in \{1,2,\ldots, m/2\}$ (even vertices have the same color), 
\item $\lambda(p_1)=\lambda(q_2)$ and $\lambda(p_2)=\lambda(q_1)$. 
\end{enumerate}
Then, $\operatorname{det}(K_P) = \operatorname{det}(K_Q)$ even though $Q$ is neither identical nor the reflection of $P$ for $m>2$.
\end{theorem}
\begin{proof} 
The paths have the following structure:
    \begin{center}
    P: $\quad$
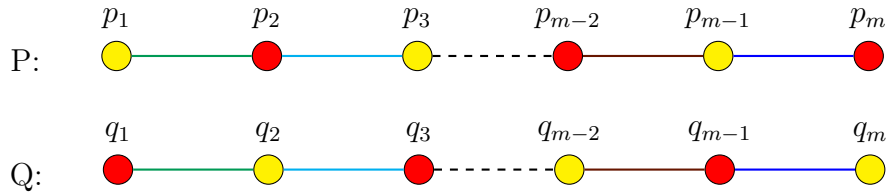
    \begin{tikzpicture}[scale=1]
        \node[draw, circle, fill=yellow] (1) at (0, 0) {};
        \node[draw, circle, fill=red] (2) at (2, 0) {};
        \node[draw, circle, fill=yellow] (3) at (4, 0) {};
        \node[draw, circle, fill=red] (n-2) at (6, 0) {};
        \node[draw, circle, fill=yellow] (n-1) at (8, 0) {};
        \node[draw, circle, fill=red] (n) at (10, 0) {};

        \draw[ForestGreen, thick] (1) -- (2);
        \draw[ProcessBlue, thick] (2) -- (3);
        \draw[black, thick, dashed] (3) -- (n-2);
        \draw[Sepia, thick] (n-2) -- (n-1);
        \draw[blue, thick] (n-1) -- (n);

        \node at (0,0.5) {$p_1$};
        \node at (2,0.5) {$p_2$};
        \node at (4,0.5) {$p_3$};
        \node at (6,0.5) {$p_{m-2}$};
        \node at (8,0.5) {$p_{m-1}$};
        \node at (10,0.5) {$p_m$};
        
    \end{tikzpicture}
    
    \vspace{0.5cm}
    
    Q: $\quad$
    \begin{tikzpicture}[scale=1]
        \node[draw, circle, fill=red] (1) at (0, 0) {};
        \node[draw, circle, fill=yellow] (2) at (2, 0) {};
        \node[draw, circle, fill=red] (3) at (4, 0) {};
        \node[draw, circle, fill=yellow] (n-2) at (6, 0) {};
        \node[draw, circle, fill=red] (n-1) at (8, 0) {};
        \node[draw, circle, fill=yellow] (n) at (10, 0) {};
    
        \draw[ForestGreen, thick] (1) -- (2);
        \draw[ProcessBlue, thick] (2) -- (3);
        \draw[black, thick, dashed] (3) -- (n-2);
        \draw[Sepia, thick] (n-2) -- (n-1);
        \draw[blue, thick] (n-1) -- (n);

        \node at (0,0.5) {$q_1$};
        \node at (2,0.5) {$q_2$};
        \node at (4,0.5) {$q_3$};
        \node at (6,0.5) {$q_{m-2}$};
        \node at (8,0.5) {$q_{m-1}$};
        \node at (10,0.5) {$q_m$};
    \end{tikzpicture}
    \end{center}

\noindent From the visual representation, it is clear that $Q$ is neither identical nor the reflection of $P$. We assign the odd vertices of $P$ the partial correlation $k_{11}$ and the even vertices of $P$ the partial correlation $k_{22}$. By the color constraints specified in the theorem, the concentration matrices $K_P = (p_{ij})$ and $K_Q = (q_{ij})$ meet the following conditions:
    \begin{align*}
        p_{ii} & = k_{11} \text{ and } q_{ii} = k_{22}, \quad \text{if } i \in \{1,2, \dots, m\} \text{ is odd},\\
        p_{ii} & = k_{22} \text{ and } q_{ii} = k_{11}, \quad \text{if } i \in \{1,2, \dots, m\} \text{ is even},\\
        p_{ij} &= q_{ij} = k_{ij}, \qquad \qquad \ \text{for all } \{p_i,p_j\} \in E_P \text{ and } \{q_i,q_j\} \in E_Q.
    \end{align*}
    \noindent The concentration matrices of $P$ and $Q$ are then given by:
    \[
     K_P = \begin{pmatrix}
        k_{11} & k_{12} & 0 & \cdots & 0 \\
        k_{12} & k_{22} & k_{23} & \ddots & \vdots \\
        0 & k_{23} & k_{11} & \ddots & 0 \\
        \vdots & \ddots & \ddots & \ddots & k_{m-1m} \\
        0 & \cdots & 0 & k_{m-1m} & k_{22}
    \end{pmatrix},
    \hspace{1cm}
    K_Q = \begin{pmatrix}
        k_{22} & k_{12} & 0 & \cdots & 0 \\
        k_{12} & k_{11} & k_{23} & \ddots & \vdots \\
        0 & k_{23} & k_{22} & \ddots & 0 \\
        \vdots & \ddots & \ddots & \ddots & k_{m-1m} \\
        0 & \cdots & 0 & k_{m-1m} & k_{11}
    \end{pmatrix}.
    \]
    
    \noindent To calculate the determinants of $K_P$ and $K_Q$, we apply the well-known Leibniz formula:
    \[
        \operatorname{det}(K_P)  
        = \sum_{\tau \in \mathbb{S}^m} \operatorname{sgn}(\tau) \prod_{i=1}^m p_{i \tau(i)}, 
        \qquad \quad
        \operatorname{det}(K_Q)  
        = \sum_{\tau \in \mathbb{S}^m} \operatorname{sgn}(\tau) \prod_{i=1}^m q_{i \tau(i)}.
    \]
    For the tridiagonal matrices $K_P$ and $K_Q$, any permutation $\tau$ where $\tau(i)$ maps $i$ to an index not adjacent to $i$ results in a product term that is zero, since the corresponding matrix entries $p_{i \tau(i)}$ and  $q_{i \tau(i)}$ are zero for non-adjacent indices. Consequently, only permutations that are either the identity or disjoint compositions of 2-cycles, which permute adjacent indices, contribute to the determinant. Let $\mathcal{T} \subseteq \mathbb{S}^m$ denote the set of such permutations. Then:
    \[
        \operatorname{det}(K_P)  
        = \sum_{\tau \in \mathcal{T}} \operatorname{sgn}(\tau) \prod_{i=1}^m p_{i \tau(i)},
        \qquad \quad
        \operatorname{det}(K_Q)  
        = \sum_{\tau \in \mathcal{T}} \operatorname{sgn}(\tau) \prod_{i=1}^m q_{i \tau(i)}.
    \]
    The products in these formulas can be decomposed into two components: the product of variables corresponding to indices permuted under the permutation $\tau$, and the product of variables corresponding to indices fixed by $\tau$:
    \[
        \operatorname{det}(K_P)  
        = \sum_{\tau \in \mathcal{T}} \operatorname{sgn}(\tau) \prod_{\substack{i=1, \\ \tau(i) \neq i}}^m p_{i \tau(i)} \prod_{\substack{i=1, \\ \tau(i) = i}}^m p_{i \tau(i)},
        \qquad \quad
        \operatorname{det}(K_Q)  
        = \sum_{\tau \in \mathcal{T}} \operatorname{sgn}(\tau) \prod_{\substack{i=1, \\ \tau(i) \neq i}}^m q_{i \tau(i)} \prod_{\substack{i=1, \\ \tau(i) = i}}^m q_{i \tau(i)}.
    \]
    Let $t_{\tau} \in \{0,1, \dots, \frac{m}{2}\}$ denote the number of 2-cycles within the permutation $\tau \in \mathcal{T}$.
    The first product in $\operatorname{det}(K_P)$ corresponds to the off-diagonal entries of $K_P$. Any $\tau \in \mathcal{T}$ permutes an even amount of $2 t_{\tau}$ off-diagonal entries, which are equal in $K_P$ and $K_Q$. 
    Since $m$ is even, $\tau$ fixes an even amount of diagonal entries. 
    Specifically, $\tau$ fixes $\frac{m}{2} - t_{\tau}$ diagonal entries corresponding to even indices $i$, and $\frac{m}{2} - t_{\tau}$ diagonal entries corresponding to odd indices $i$. Thus, the second product in the determinant formula includes $\frac{m}{2} - t_{\tau}$ factors of $k_{11}$ and $\frac{m}{2} - t_{\tau}$ factors of $k_{22}$.
    Thus, the determinant of $K_P$ simplifies as:
    \[
        \operatorname{det}(K_P)  
        = \sum_{\tau \in \mathcal{T}} \operatorname{sgn}(\tau) \ k_{11}^{\frac{m}{2}-t_{\tau}} \ k_{22}^{\frac{m}{2}-t_{\tau}} \prod_{\substack{i=1, \\ \tau(i) \neq i}}^m k_{i \tau(i)}.
    \]
    Similarly, for $K_Q$, the fixed diagonal entries contribute $\frac{m}{2} - t_{\tau}$ factors of $k_{22}$ and $\frac{m}{2} - t_{\tau}$ factors of $k_{11}$ for each $\tau \in \mathcal{T}$, giving us:
    \[
        \operatorname{det}(K_Q)  
        = \sum_{\tau \in \mathcal{T}} \operatorname{sgn}(\tau) \ k_{22}^{\frac{m}{2}-t_{\tau}} \ k_{11}^{\frac{m}{2}-t_{\tau}} \prod_{\substack{i=1, \\ \tau(i) \neq i}}^m k_{i \tau(i)}.
    \]
    Hence, we can conclude that $\det(K_P)=\det(K_Q)$.
\end{proof}

\noindent We illustrate the above construction with the following example.
\begin{example}
    Let the two paths \(P\) and \(Q\) be colored as follows:
    \begin{center}
    P: $\quad$
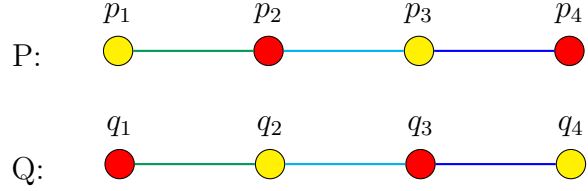
    \begin{tikzpicture}[scale=1]
        \node[draw, circle, fill=yellow] (1) at (0, 0) {};
        \node[draw, circle, fill=red] (2) at (2, 0) {};
        \node[draw, circle, fill=yellow] (3) at (4, 0) {};
        \node[draw, circle, fill=red] (4) at (6, 0) {};

        \draw[ForestGreen, thick] (1) -- (2);
        \draw[ProcessBlue, thick] (2) -- (3);
        \draw[blue, thick] (3) -- (4);
        
        \node at (0,0.5) {$p_1$};
        \node at (2,0.5) {$p_2$};
        \node at (4,0.5) {$p_3$};
        \node at (6,0.5) {$p_{4}$};
        
    \end{tikzpicture}
    
    \vspace{0.5cm}
    
    Q: $\quad$
    \begin{tikzpicture}[scale=1]
        \node[draw, circle, fill=red] (1) at (0, 0) {};
        \node[draw, circle, fill=yellow] (2) at (2, 0) {};
        \node[draw, circle, fill=red] (3) at (4, 0) {};
        \node[draw, circle, fill=yellow] (4) at (6, 0) {};

        \draw[ForestGreen, thick] (1) -- (2);
        \draw[ProcessBlue, thick] (2) -- (3);
        \draw[blue, thick] (3) -- (4);
        
        \node at (0,0.5) {$q_1$};
        \node at (2,0.5) {$q_2$};
        \node at (4,0.5) {$q_3$};
        \node at (6,0.5) {$q_{4}$};
 
    \end{tikzpicture}
    \end{center}
    Assign the color yellow the coefficient \(k_{11}\) and red the coefficient \(k_{22}\). Computing the determinant of $K_P$ gives us the following:
\begin{eqnarray*}
\det(K_P)&=& sgn(id)\prod_{i=1}^4p_{i(id)(i)} + sgn(12)\prod_{i=1}^4p_{i(12)(i)} + sgn(23)\prod_{i=1}^4p_{i(23)(i)} + \\
&&sgn(34)\prod_{i=1}^4p_{i(34)(i)} + sgn((12)(34))\prod_{i=1}^4p_{i(12)(34)(i)} \\
&=& k_{11}^2 k_{22}^2 - k_{12}^2 k_{11} k_{22} - k_{23}^2 k_{11} k_{22}- k_{34}^2 k_{11} k_{22} + k_{12}^2 k_{34}^2.
\end{eqnarray*}
    
As $P$ and $Q$ have the same edge colors and alternating vertex colors, the same monomial is obtained for every disjoint union of $2$-cycle permutations in $\det(K_P)$ and $\det(K_Q)$. This shows that the two determinants are indeed equal. This coloring configuration can be interpreted as a local reflection of each edge.
\end{example}

\begin{theorem}
\label{theorem:PathDetTypeOdd}
Let $P$ and $Q$ be two colored paths on $m$ vertices with vertex sets $V_P = \{p_1, p_2, \dots, p_m\}$ and $V_Q = \{q_1, q_2, \dots, q_m\}$, and edge sets \(E_P = \{ \{p_i,p_{i+1}\} \mid i \in \{1,2, \dots, m-1\} \}\) and \(E_Q = \{ \{q_i,q_{i+1}\} \mid i, \in \{1,2, \dots, m-1\} \}\), respectively. Let $m$ be odd and the coloring of $P$ and $Q$ satisfy the following conditions:
\begin{enumerate}
\item $\lambda(p_1)=\lambda(p_{2n+1})= \lambda(q_1)=\lambda(q_{2n+1})$ for every $n\in \{1,2,\ldots, (m-1)/2\}$, 
\item $\lambda(p_{2n})=\lambda(q_{2n})$ for every $n\in \{1,2,\ldots,(m-1)/2\}$,
\item $\lambda(\{p_i,p_{i+1}\}) = \lambda(\{q_{j},q_{j+1}\})$ and $\lambda(\{p_j,p_{j+1}\}) = \lambda(\{q_{i},q_{i+1}\})$, for all odd $i \in \{1,2, \dots, m\}$ and all even $j \in \{1,2, \dots, m\}$.
\end{enumerate}
Then, $\operatorname{det}(K_P) = \operatorname{det}(K_Q)$ even though $Q$ is neither identical nor the reflection of $P$ for $m>3$.
\end{theorem}
\begin{proof} 
    The paths have the following structure:
    \begin{center}
    P: $\quad$
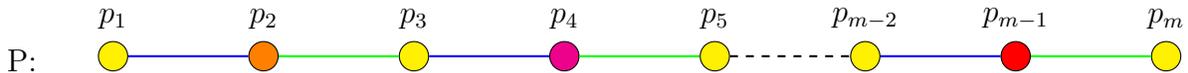
    \begin{tikzpicture}[scale=1]
        \node[draw, circle, fill=yellow] (1) at (0, 0) {};
        \node[draw, circle, fill=orange] (2) at (2, 0) {};
        \node[draw, circle, fill=yellow] (3) at (4, 0) {};
        \node[draw, circle, fill=magenta] (4) at (6, 0) {};
        \node[draw, circle, fill=yellow] (5) at (8, 0) {};
        \node[draw, circle, fill=yellow] (n-2) at (10, 0) {};
        \node[draw, circle, fill=red] (n-1) at (12, 0) {};
        \node[draw, circle, fill=yellow] (n) at (14, 0) {};

        \draw[blue, thick] (1) -- (2);
        \draw[green, thick] (2) -- (3);
        \draw[blue, thick] (3) -- (4);
        \draw[green, thick] (4) -- (5);
        \draw[black, thick, dashed] (5) -- (n-2);
        \draw[blue, thick] (n-2) -- (n-1);
        \draw[green, thick] (n-1) -- (n);

        \node at (0,0.5) {$p_1$};
        \node at (2,0.5) {$p_2$};
        \node at (4,0.5) {$p_3$};
        \node at (6,0.5) {$p_{4}$};
        \node at (8,0.5) {$p_{5}$};
        \node at (10,0.5) {$p_{m-2}$};
        \node at (12,0.5) {$p_{m-1}$};
        \node at (14,0.5) {$p_m$};
        
    \end{tikzpicture}
    
    \vspace{0.5cm}
    
    Q: $\quad$
    
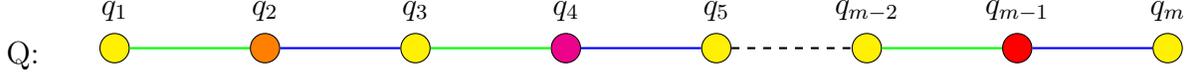
\begin{tikzpicture}[scale=1]
        \node[draw, circle, fill=yellow] (1) at (0, 0) {};
        \node[draw, circle, fill=orange] (2) at (2, 0) {};
        \node[draw, circle, fill=yellow] (3) at (4, 0) {};
        \node[draw, circle, fill=magenta] (4) at (6, 0) {};
        \node[draw, circle, fill=yellow] (5) at (8, 0) {};
        \node[draw, circle, fill=yellow] (n-2) at (10, 0) {};
        \node[draw, circle, fill=red] (n-1) at (12, 0) {};
        \node[draw, circle, fill=yellow] (n) at (14, 0) {};
    
        \draw[green, thick] (1) -- (2);
        \draw[blue, thick] (2) -- (3);
        \draw[green, thick] (3) -- (4);
        \draw[blue, thick] (4) -- (5);
        \draw[black, thick, dashed] (5) -- (n-2);
        \draw[green, thick] (n-2) -- (n-1);
        \draw[blue, thick] (n-1) -- (n);

        \node at (0,0.5) {$q_1$};
        \node at (2,0.5) {$q_2$};
        \node at (4,0.5) {$q_3$};
        \node at (6,0.5) {$q_{4}$};
        \node at (8,0.5) {$q_{5}$};
        \node at (10,0.5) {$q_{m-2}$};
        \node at (12,0.5) {$q_{m-1}$};
        \node at (14,0.5) {$q_m$};
    \end{tikzpicture}
    \end{center}
    By the color constraints specified in the theorem, the concentration matrices $K_P = (p_{ij})$ and $K_Q = (q_{ij})$ meet the following conditions:
    \begin{align*}
        p_{ii} & = q_{ii} = k_{ii}, \qquad \ \qquad \qquad \text{for all even } i \in \{1,2, \dots, m\},\\
        p_{ii} & = p_{jj} = q_{ii} = q_{jj} = k_{11}, \quad \text{for all odd } i,j \in \{1,2, \dots, m\},\\
        p_{ii+1} &= q_{i+1i+2} = k_{12}, \qquad \qquad \ \text{for all odd } i \in \{1,2, \dots, m-2\},\\
        p_{ii+1} &= q_{i+1i+2} = k_{23}, \qquad \qquad \ \text{for all even } i \in \{1,2, \dots, m-1\} \operatorname{mod} m-1.
    \end{align*}
    \noindent The concentration matrices of $P$ and $Q$ are therefore given by:
    \[
     K_P = \begin{pmatrix}
        k_{11} & k_{12} & 0 & 0& \cdots & \cdots & 0 \\
        k_{12} & k_{22} & k_{23} & 0 & \cdots & \cdots & \vdots \\
        0 & k_{23} & k_{11} & k_{12} & 0 & \cdots & 0 \\
        0 & 0 & k_{12} & k_{44} & k_{23} & \cdots & 0 \\        
        \vdots & \ddots & \ddots & \ddots & \ddots &  \ddots & \vdots \\
        \vdots & \ddots & \ddots & \ddots & \ddots &  \ddots & k_{23} \\
        0 & \cdots & \cdots & \cdots & 0 & k_{23} & k_{11}
    \end{pmatrix},
    \hspace{1cm}
    K_Q = \begin{pmatrix}
        k_{11} & k_{23} & 0 & 0& \cdots & \cdots & 0 \\
        k_{23} & k_{22} & k_{12} & 0 & \cdots & \cdots & \vdots \\
        0 & k_{12} & k_{11} & k_{23} & 0 & \cdots & 0 \\
        0 & 0 & k_{23} & k_{44} & k_{12} & \cdots & 0 \\        
        \vdots & \ddots & \ddots & \ddots & \ddots &  \ddots & \vdots \\
        \vdots & \ddots & \ddots & \ddots & \ddots &  \ddots & k_{12} \\
        0 & \cdots & \cdots & \cdots & 0 & k_{12} & k_{11}
    \end{pmatrix}.
    \]
We prove this theorem by applying induction on $m$, the number of vertices in $P$ and $Q$. For $m=3$, observe that $Q$ is actually the reflection of $P$, and thus we know that $\det(K_P)=\det(K_Q)$. Similarly, for $m=5$, we get the matrices $K_P$ and $K_Q$ as follows:
\[ K_P =
    \left(
    \begin{NiceMatrix}[create-large-nodes]
    k_{11} & k_{12} & 0      & 0      & 0 \\
    k_{12} & k_{22} & k_{23} & 0      & 0 \\
    0      & k_{23} & k_{11} & k_{12} & 0 \\
    0      & 0      & k_{12} & k_{44} & k_{23} \\
    0      & 0      & 0      & k_{23} & k_{11}
    \CodeAfter
      \tikz {
        \draw[black, thick] (1-1.north west) rectangle (3-3.south east);
        \node at ($(1-1.north west) + (1.15,0.25)$) {\textbf{$K_{P-2}$}};
      }
    \end{NiceMatrix}
    \right),
    \hspace{1cm}
     K_Q =
    \left(
    \begin{NiceMatrix}[create-large-nodes]
    k_{11} & k_{23} & 0      & 0      & 0 \\
    k_{23} & k_{22} & k_{12} & 0      & 0 \\
    0      & k_{12} & k_{11} & k_{23} & 0 \\
    0      & 0      & k_{23} & k_{44} & k_{12} \\
    0      & 0      & 0      & k_{12} & k_{11}
    \CodeAfter
      \tikz {
        \draw[black, thick] (1-1.north west) rectangle (3-3.south east);
        \node at ($(1-1.north west) + (1.15,0.25)$) {\textbf{$K_{Q-2}$}};
      }
    \end{NiceMatrix}
    \right).
    \]
We state the matrices for $m=5$ for the purpose of clarity. The idea here is to expand the determinant through the last row and write the expression in terms of the determinants of $K_{P-2}$ and $K_{P-4}$, where $P-2$ and $P-4$ are the colored subpaths of $P$ induced on the first $m-2$ and $m-4$ vertices, respectively. For $m=5$, the determinants of $K_P$ and $K_Q$ can be written as:
\begin{eqnarray*}
\det(K_P) 
    &=& k_{23}^2 \operatorname{det}(K_{P-2}) - k_{11} k_{44} \operatorname{det}(K_{P-2}) + k_{11} k_{12}^2 \operatorname{det} 
    \begin{pmatrix}
        k_{11} & k_{12} \\
        k_{12} & k_{22} \\ 
    \end{pmatrix}, \text{ and} \\
   \det(K_Q) &=& k_{12}^2 \operatorname{det}(K_{Q-2}) - k_{11} k_{44} \operatorname{det}(K_{Q-2}) + k_{11} k_{23}^2 \operatorname{det} 
    \begin{pmatrix}
        k_{11} & k_{23} \\
        k_{23} & k_{22} \\
    \end{pmatrix}. 
    \end{eqnarray*}
As $\det(K_{P-2})=\det(K_{Q-2})$ (as the statement is true for $m=3$), subtracting $\det(K_Q)$ from $\det(K_P)$ and interchanging $\det(K_{P-2})$ and $\det(K_{Q-2})$ gives us:
\begin{eqnarray*}
\det(K_P) - \det(K_Q) &=& 
k_{23}^2 (\operatorname{det}(K_{Q-2}) - k_{11} \operatorname{det}\begin{pmatrix}
        k_{11} & k_{23} \\
        k_{23} & k_{22} \\
    \end{pmatrix}) 
    -  \\
    &&k_{12}^2 (\operatorname{det}(K_{P-2}) - k_{11} \operatorname{det} 
    \begin{pmatrix}
        k_{11} & k_{12} \\
        k_{12} & k_{22} \\
    \end{pmatrix}) \\
    &=& k_{23}^2(\det(K_{Q-2})-k_{11}\det(K_{Q-3}))-k_{12}^2(\det(K_{P-2})-k_{11}\det(K_{P-3})) \\
   &=& k_{23}^2(k_{12}^2 \det(K_{Q-4})) - k_{12}^2(k_{23}^2\det(K_{P-4})) \\
   &=& 0,
\end{eqnarray*}
as $\det(K_{P-4})$ is also equal to $\det(K_{Q-4})$. Now, assuming the induction hypothesis on $m-2$ and $m-4$, we use the exact same technique as above to prove the induction statement. We have,
\begin{eqnarray*}
    \det(K_P)&=& k_{23}^2 \det(K_{P-2})- k_{11}k_{m-1m-1}\det(K_{P-2}) +k_{11}k_{12}^2\det(K_{P-3}), \text{ and} \\
    \det(K_Q)&=& k_{12}^2 \det(K_{Q-2})- k_{11}k_{m-1m-1}\det(K_{Q-2}) +k_{11}k_{23}^2\det(K_{Q-3}).
\end{eqnarray*}
As $\det(K_{P-2})$ is equal to $\det(K_{Q-2})$ by the induction hypothesis, interchanging them gives us,
\begin{eqnarray*}
    \det(K_P)-\det(K_Q)=&=& k_{23}^2(\det(K_{Q-2})-k_{11}\det(K_{Q-3}))-k_{12}^2(\det(K_{P-2})-k_{11}\det(K_{P-3})) \\
   &=& k_{23}^2(k_{12}^2 \det(K_{Q-4})) - k_{12}^2(k_{23}^2\det(K_{P-4})) \\
   &=& 0,
\end{eqnarray*}
as $\det(K_{Q-4})=\det(K_{P-4})$ by the induction hypothesis, thus concluding the proof.
\end{proof}

\noindent We illustrate the above configuration with the following example.
\begin{example}
    Let \(P\) and \(Q\) be two paths colored as follows:
        \begin{center}
    P: $\quad$
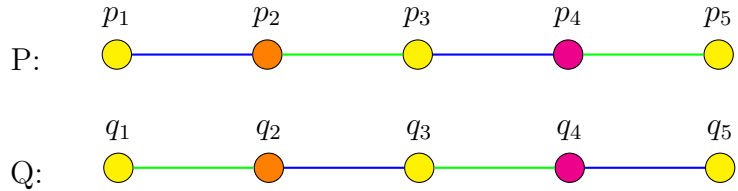
    \begin{tikzpicture}[scale=1]
        \node[draw, circle, fill=yellow] (1) at (0, 0) {};
        \node[draw, circle, fill=orange] (2) at (2, 0) {};
        \node[draw, circle, fill=yellow] (3) at (4, 0) {};
        \node[draw, circle, fill=magenta] (4) at (6, 0) {};
        \node[draw, circle, fill=yellow] (5) at (8, 0) {};
        
        \draw[blue, thick] (1) -- (2);
        \draw[green, thick] (2) -- (3);
        \draw[blue, thick] (3) -- (4);
        \draw[green, thick] (4) -- (5);

        \node at (0,0.5) {$p_1$};
        \node at (2,0.5) {$p_2$};
        \node at (4,0.5) {$p_3$};
        \node at (6,0.5) {$p_{4}$};
        \node at (8,0.5) {$p_{5}$};
        
    \end{tikzpicture}
    
    \vspace{0.5cm}
    
    Q: $\quad$
    \begin{tikzpicture}[scale=1]
        \node[draw, circle, fill=yellow] (1) at (0, 0) {};
        \node[draw, circle, fill=orange] (2) at (2, 0) {};
        \node[draw, circle, fill=yellow] (3) at (4, 0) {};
        \node[draw, circle, fill=magenta] (4) at (6, 0) {};
        \node[draw, circle, fill=yellow] (5) at (8, 0) {};
    
        \draw[green, thick] (1) -- (2);
        \draw[blue, thick] (2) -- (3);
        \draw[green, thick] (3) -- (4);
        \draw[blue, thick] (4) -- (5);

        \node at (0,0.5) {$q_1$};
        \node at (2,0.5) {$q_2$};
        \node at (4,0.5) {$q_3$};
        \node at (6,0.5) {$q_{4}$};
        \node at (8,0.5) {$q_{5}$};
    \end{tikzpicture}
    \end{center}
Observe that $P$ and $Q$ satisfy the color configuration as stated in Theorem \ref{theorem:PathDetTypeOdd}. Computing the determinant of $K_P$ and $K_Q$ gives us, 
\begin{eqnarray*}
\det(K_P)= \det(K_Q) &=&  k_{11}^3 k_{22} k_{44} - k_{12}^2 k_{11}^2 k_{44} -k_{23}^2 k_{11}^2 k_{44} - k_{12}^2 k_{11}^2 k_{22} -k_{23}^2 k_{11}^2 k_{22}
    + \\
    &&k_{12}^4 k_{11} + k_{12}^2 k_{23}^2 k_{11} + k_{23}^4 k_{11}.
\end{eqnarray*}
This coloring configuration can be interpreted as local reflection of two edges at a time.
\end{example}

The purpose of constructing these two path color configurations will be evident in the next section, where we use these configurations to construct the counterexamples to the conjecture.

\section{Counterexamples to Conjecture \ref{Conj:ifandonlyif} and its potential strengthening}\label{section:counterexamples and strengthening}
In this section, we construct counterexamples to Conjecture \ref{Conj:ifandonlyif} by using the non trivial path color configurations obtained in Section \ref{section:analysis of path graphs}. We also analyze the potential ways to strengthen the conjecture. We begin with counterexamples on an $8$-cycle and a $10$-cycle. The constructions provided below illustrate a general technique for generating counterexamples to the conjecture.

\begin{example}\label{CounterexampleConj6}
Let $\cg_1$ be the $8$-cycle as shown in Figure \ref{figure:8 cycle counterexample}.
This cycle is constructed by using the path color configuration described in Theorem \ref{theorem:pathDetTypeEven} to obtain the linear binomial $\sigma_{14}-\sigma_{58}$. Specifically, we draw the paths $1\leftrightarrow 4$ and $5\leftrightarrow 8$ using the color configuration obtained in Theorem \ref{theorem:pathDetTypeEven}, and join the two paths by edges $\{8,1\}$ and $\{4,5\}$.
By Theorem \ref{theorem:pathDetTypeEven}, we know that
    \[
    \det\big(K_{ \setminus 1 \leftrightarrow 4}\big) = \det\big(K_{ \setminus 5 \leftrightarrow 8}\big).
    \]
Furthermore, the relevant determinants for the complementary contributions to the covariance entries correspond to the paths \(\{2-3\}\) and \(\{6-7\}\). These are reflections of each other, guaranteeing that
\[
    \det\big(K_{ \setminus 1 \overset{c}{ \leftrightarrow}4}\big) = \det\big(K_{ \setminus 5 \overset{c}{\leftrightarrow}8}\big).
\]
The construction of the cycle also ensures that the edge colors along the shorter and complementary paths of $\sigma_{14}$ and $\sigma_{58}$ are identical. 
The direct computation of the vanishing ideal indeed confirms $\sigma_{14}-\sigma_{58}\in I_{\cg_1}$. 
Now, observe that the due to the vertex color configuration, the only potential symmetry in this \(8\)-cycle is the reflection along the axis passing through the edges $\{8,1\}$ and $\{4,5\}$. 
However, this reflection maps edge $\{1,2\}$ to $\{7,8\}$, which differ in color, and hence, it is not a valid graph symmetry. Therefore, $\cg_1$ has no graph symmetry, giving us the first counterexample to Conjecture~\ref{Conj:ifandonlyif}.

Similarly, let \(\cg_2\) be color $10$-cycle as shown in Figure \ref{figure:10 cycle counterexample}.
This cycle is constructed in way to obtain the linear binomial \(\sigma_{15}-\sigma_{610}\) in the vanishing ideal \(I_{\cg_2}\) while using the color configuration from Theorem~\ref{theorem:PathDetTypeOdd}.
The paths on vertices \(V \big( \cg_2 \setminus 1 \leftrightarrow 5 \big)\) and \(V \big( \cg_2 \setminus 6 \leftrightarrow 10 \big)\) satisfy the conditions defined in Theorem~\ref{theorem:PathDetTypeOdd}. The paths on vertices \(V \big( \cg_2 \setminus 1 \overset{c}{ \leftrightarrow}5 \big)\) and \(V \big( \cg_2 \setminus 6\overset{c}{ \leftrightarrow }10 \big)\) also satisfy the configuration, ensuring that 
\begin{align*}
\det\big(K_{ \setminus 1 \leftrightarrow 5}\big) &= \det\big(K_{ \setminus 6 \leftrightarrow 10}\big), \ \text{and}\\
\det\big(K_{ \setminus 1 \overset{c}{\leftrightarrow}5}\big) &= \det\big(K_{ \setminus 6 \overset{c}{\leftrightarrow}10}\big).
\end{align*}
Since the edge colors along the shorter and complementary paths, respectively, are also equal, by Corollary~\ref{Cor:J&WforCycle} we know that \(\sigma_{15}-\sigma_{610} \in I_{\cg_2}\), even though there is no graph symmetry in $\cg_2$. 
\begin{figure}[H]
\centering
\begin{minipage}{0.45\textwidth}
\begin{tikzpicture}[scale=0.8, every node/.style={circle, draw, inner sep=1pt, minimum size=4mm, font=\small}]
    \node[fill=red, label=above:{1}] (1) at (90:2) {};
    \node[fill=yellow, label=above right:{2}] (2) at (45:2) {};
    \node[fill=red, label=right:{3}] (3) at (0:2) {};
    \node[fill=yellow, label=below right:{4}] (4) at (-45:2) {};
    \node[fill=yellow, label=below:{5}] (5) at (-90:2) {};
    \node[fill=red, label=below left:{6}] (6) at (-135:2) {};
    \node[fill=yellow, label=left:{7}] (7) at (180:2) {};
    \node[fill=red, label=above left:{8}] (8) at (135:2) {};

    \draw[brown, thick] (1) -- (2);
    \draw[green, thick] (2) -- (3);
    \draw[blue, thick] (3) -- (4);
    \draw[black, thick] (4) -- (5);
    \draw[brown, thick] (5) -- (6);
    \draw[green, thick] (6) -- (7);
    \draw[blue, thick] (7) -- (8);
    \draw[black, thick] (8) -- (1);

    \draw[dashed, thick] (-0.9,2.3) -- (0.9,-2.3);
\end{tikzpicture}
\caption{$\cg_1$: Colored $8$-cycle counterexample}\label{figure:8 cycle counterexample}
\label{fig:8cycle}
\end{minipage}
\centering
    \begin{minipage}{0.45\textwidth}
        \centering
        \begin{tikzpicture}[scale=0.8, every node/.style={circle, draw, inner sep=1pt, minimum size=4mm, font=\small}]
            \node[fill=yellow, label=above:{1}] (1) at (90:2.5) {};
            \node[fill=orange, label=above right:{2}] (2) at (54:2.5) {};
            \node[fill=yellow, label=right:{3}] (3) at (18:2.5) {};
            \node[fill=purple, label=below right:{4}] (4) at (-18:2.5) {};
            \node[fill=yellow, label=below:{5}] (5) at (-54:2.5) {};
            \node[fill=yellow, label=below:{6}] (6) at (-90:2.5) {};
            \node[fill=orange, label=below left:{7}] (7) at (-126:2.5) {};
            \node[fill=yellow, label=left:{8}] (8) at (-162:2.5) {};
            \node[fill=purple, label=above left:{9}] (9) at (162:2.5) {};
            \node[fill=yellow, label=above:{10}] (10) at (126:2.5) {};
            
            \draw[blue, thick] (1) -- (2);
            \draw[green, thick] (2) -- (3);
            \draw[blue, thick] (3) -- (4);
            \draw[green, thick] (4) -- (5);
            \draw[brown, thick] (5) -- (6);
            \draw[green, thick] (6) -- (7);
            \draw[blue, thick] (7) -- (8);
            \draw[green, thick] (8) -- (9);
            \draw[blue, thick] (9) -- (10);
            \draw[brown, thick] (10) -- (1);
        \end{tikzpicture}
        \caption{$\cg_2$: Colored $10$-cycle counterexample}
        \label{figure:10 cycle counterexample}
    \end{minipage}%
\end{figure}
\end{example}

Notice that in the above example for $\cg_1$, the existence of the paths $\{1-2-3-4\}$ and $\{5-6-7-8\}$ satisfying the conditions in Theorem \ref{theorem:pathDetTypeEven} was crucial to obtain the counterexample. This is because, if the path $\{1-2-3-4\}$ was identical or a reflection of $\{5-6-7-8\}$, it would have resulted in the existence of a graph symmetry, and hence, not violating the conjecture. 
A similar argument applies to the above counterexample $\cg_2$.

Applying the technique illustrated in Example~\ref{CounterexampleConj6} always yields cycles of even length. However, our computational study also provided us with counterexamples to Conjecture~\ref{Conj:ifandonlyif} which did not use the non-trivial path color configurations from Theorem \ref{theorem:pathDetTypeEven} and \ref{theorem:PathDetTypeOdd}. The following example illustrates this, and in particular provides a counterexample on a cycle of odd length.

\begin{example}
Let $\cg_1$ be the $4$-cycle shown in Figure \ref{Figure:4cycleCounter}. The paths on vertices $3 \leftrightarrow 4$ and $1 \leftrightarrow 2$ are identical, thereby confirming that 
\[
\operatorname{det}(K_{\setminus 1 \leftrightarrow 2}) 
=\operatorname{det}(K_{\setminus 3 \leftrightarrow 4}). 
\]
Therefore, the linear binomial $\sigma_{12} - \sigma_{34}$ lies in $I_{\cg_2}$. However, there is no symmetry in this $4$-cycle due to its specific coloring.

We also get a counterexample of a colored cycle $\cg_2$ of odd size, as shown in Figure \ref{figure:odd cycle counterexample}, confirming that the conjecture is not true for cycles of any size.
The paths on vertices \(V \big( \cg_2 \setminus 1 \leftrightarrow 3 \big)\) and \(V \big( \cg_2 \setminus 7 \leftrightarrow 9 \big)\)
are identical. Likewise, the ones on \(V \big( \cg_2 \setminus 1 \overset{c}{ \leftrightarrow} 3 \big)\) and \(V \big( \cg_2 \setminus 7 \overset{c}{ \leftrightarrow} 9 \big)\)
are also equal. As a result, we get that 
\begin{align*}
    \operatorname{det}(K_{\setminus 1 \leftrightarrow 3}) &= \operatorname{det}(K_{\setminus 7 \leftrightarrow 9}), \ \text{and} \\
    \operatorname{det}(K_{\setminus 1 \overset{c}{\leftrightarrow} 3}) &= \operatorname{det}(K_{\setminus 7 \overset{c}{\leftrightarrow} 9}),
    \end{align*}
implying that the linear binomial $\sigma_{13}-\sigma_{79}$ lies in the vanishing ideal \(I_{\cg_2}\).
However, there is no symmetry in this $9$-cycle due to the existence of the unique black edge $\{9,1\}$ and the way the other edges are colored, thereby contradicting Conjecture \ref{Conj:ifandonlyif}.
\begin{figure}[H]
    \begin{minipage}{0.45\textwidth}
    \centering
    \begin{tikzpicture}[scale=0.75]
        \node[draw, circle, fill=red] (1) at (0, 2) {};
        \node[draw, circle, fill=yellow] (2) at (2, 2) {};
        \node[draw, circle, fill=red] (3) at (2, 0) {};
        \node[draw, circle, fill=yellow] (4) at (0, 0) {};
        
        \draw[thick, blue] (1) -- (2);
        \draw[thick, brown] (2) -- (3);
        \draw[thick, blue] (3) -- (4);
        \draw[thick, green] (4) -- (1);
    
        \node at (-0.5, 2.6) {1};
        \node at (2.5, 2.6) {2};
        \node at (2.5, -0.6) {3};
        \node at (-0.5, -0.6) {4};
    \end{tikzpicture}
    \caption{$\cg_1$: Colored 4-cycle counterexample}
    \label{Figure:4cycleCounter}
    \end{minipage}%
    \begin{minipage}{0.45\textwidth}
        \centering
        \begin{tikzpicture}[scale=0.9, every node/.style={circle, draw, inner sep=1pt, minimum size=4mm, font=\small}]
            \node[fill=red, label=above:{1}] (1) at (90:2.2) {};
            \node[fill=yellow, label=above right:{2}] (2) at (50:2.2) {};
            \node[fill=red, label=right:{3}] (3) at (10:2.2) {};
            \node[fill=red, label=below right:{4}] (4) at (-30:2.2) {};
            \node[fill=yellow, label=below:{5}] (5) at (-70:2.2) {};
            \node[fill=red, label=below left:{6}] (6) at (-110:2.2) {};
            \node[fill=red, label=left:{7}] (7) at (-150:2.2) {};
            \node[fill=yellow, label=above left:{8}] (8) at (-190:2.2) {};
            \node[fill=red, label=above:{9}] (9) at (-230:2.2) {};
        
            \draw[blue, thick] (1) -- (2);
            \draw[green, thick] (2) -- (3);
            \draw[brown, thick] (3) -- (4);
            \draw[blue, thick] (4) -- (5);
            \draw[green, thick] (5) -- (6);
            \draw[brown, thick] (6) -- (7);
            \draw[blue, thick] (7) -- (8);
            \draw[green, thick] (8) -- (9);
            \draw[black, thick] (9) -- (1);
        \end{tikzpicture}
        \caption{$\cg_2$: Colored odd-cycle counterexample}\label{figure:odd cycle counterexample}
    \end{minipage}
\end{figure}
\end{example}


In the above example for $\cg_2$, the paths $\{4-5-6-7-8-9\}$ (corresponding to $\det(K_{1\overset{c}\leftrightarrow 3 \setminus \{1,3\}})$) and $\{1-2-3-4-5-6\}$ (corresponding to $\det(K_{7\overset{c}\leftrightarrow 9 \setminus \{7,9\}})$) are equal to each other. At the same time, the unique black edge $\{9,1\}$ is not included in any of the path determinants present in the image of $\sigma_{13}$ and $\sigma_{79}$ , i.e., $\det(K_{1\leftrightarrow 3 \setminus \{1,3\}}), \det(K_{1\overset{c}\leftrightarrow 3 \setminus \{1,3\}}), \det(K_{7\leftrightarrow 9 \setminus \{7,9\}})$ and $\det(K_{7\overset{c}\leftrightarrow 9 \setminus \{7,9\}})$. This gives us an important insight that the conjecture can also be violated without using the color configurations obtained from Theorem \ref{theorem:pathDetTypeEven} and \ref{theorem:PathDetTypeOdd}. 

We have now shown that we can construct counterexamples to Conjecture \ref{Conj:ifandonlyif} on even and odd cycles. 
This clearly implies that the current version of the conjecture is not true. We thus explore the possible strengthening of conditions in the conjecture and check whether they hold or not. As we have seen in Theorem \ref{Theo:UniformColoring} that the statement of the conjecture does hold for uniform coloring on $n$-cycles, a natural strengthening of the conditions would be to check if the conjecture holds for colored cycles with uniform vertex coloring or uniform edge coloring. 

\begin{definition}
Let $\cg$ be a colored graph. Then \textit{uniform vertex coloring} is any graph coloring of $\cg$ where all the vertices in $\cg$ are colored the same. Similarly, \textit{uniform edge coloring} is any graph coloring where all the edges in $\cg$ are colored the same. A graph is said to have a uniform coloring if it satisfies both uniform vertex coloring and uniform edge coloring.
\end{definition}

Observe that if we include the uniform vertex coloring or uniform edge coloring constraint in the coloring conditions obtained in Theorem \ref{theorem:pathDetTypeEven}
and \ref{theorem:PathDetTypeOdd}, then the resulting colored paths $P$ and $Q$ either become identical or reflection of each other. Thus, we cannot construct counterexamples for uniform vertex colored or uniform edge colored cycles in the similar way as seen in the previous two examples. By doing a computational study, we obtained the following counterexamples of even and odd cycles with uniform vertex coloring.

\begin{example}
\label{CounterexampleConjuniVertex6}
Let $\cg_1$ be the colored $6$-cycle as shown in Figure \ref{figure: uniform vertex even cycle}.
Computing the vanishing ideal gives us that the linear binomial $\sigma_{35}-\sigma_{26}$ lies in $I_{\cg_1}$. However, computations also confirm that there are no graph symmetries in $\cg_1$. This is because the only potential symmetry in $\cg_1$ would require a reflection along the symmetry axis passing through the edges $\{2,3\}$ and $\{5,6\}$, due to the existence of the unique blue edge. However, this is not a valid symmetry since the reflection maps the edge $\{1,6\}$ to $\{4,5\}$, which differ in color.

Similarly, let \(\cg_2\) be the colored $9$-cycle as shown in Figure \ref{figure: uniform vertex odd cycle}. A similar computation gives us that the binomial $\sigma_{13}-\sigma_{79}$ lies in $I_{\cg_2}$ even though $\cg_2$ does not have any graph symmetry. The latter can also be verified by the existence of the unique black edge $\{1,9\}$.
\begin{figure}[H]
\begin{minipage}{0.45\textwidth}
    \centering
    \begin{tikzpicture}[scale=0.65, every node/.style={circle, draw, inner sep=1pt, minimum size=4mm, font=\small}]
        \node[fill=red, label=above:{1}] (1) at (90:2) {};
        \node[fill=red, label=right:{2}] (2) at (30:2) {};
        \node[fill=red, label=below:{3}] (3) at (330:2) {};
        \node[fill=red, label=below:{4}] (4) at (270:2) {};
        \node[fill=red, label=left:{5}] (5) at (210:2) {};
        \node[fill=red, label=above left:{6}] (6) at (150:2) {};
    
        \draw[green, thick] (1) -- (2);
        \draw[green, thick] (2) -- (3);
        \draw[brown, thick] (3) -- (4);
        \draw[green, thick] (4) -- (5);
        \draw[blue, thick] (5) -- (6);
        \draw[brown, thick] (6) -- (1);
    \end{tikzpicture}
    \caption{$\cg_1$: Uniform vertex colored even cycle}\label{figure: uniform vertex even cycle}
    \end{minipage}%
    \begin{minipage}{0.45\textwidth}
    \centering
    \begin{tikzpicture}[scale=0.9, every node/.style={circle, draw, inner sep=1pt, minimum size=4mm, font=\small}]

    \node[fill=red, label=above:{1}] (1) at (90:2.2) {};
    \node[fill=red, label=above right:{2}] (2) at (50:2.2) {};
    \node[fill=red, label=right:{3}] (3) at (10:2.2) {};
    \node[fill=red, label=below right:{4}] (4) at (-30:2.2) {};
    \node[fill=red, label=below:{5}] (5) at (-70:2.2) {};
    \node[fill=red, label=below left:{6}] (6) at (-110:2.2) {};
    \node[fill=red, label=left:{7}] (7) at (-150:2.2) {};
    \node[fill=red, label=above left:{8}] (8) at (-190:2.2) {};
    \node[fill=red, label=above:{9}] (9) at (-230:2.2) {};

    \draw[blue, thick] (1) -- (2);
    \draw[green, thick] (2) -- (3);
    \draw[brown, thick] (3) -- (4);
    \draw[blue, thick] (4) -- (5);
    \draw[green, thick] (5) -- (6);
    \draw[brown, thick] (6) -- (7);
    \draw[blue, thick] (7) -- (8);
    \draw[green, thick] (8) -- (9);
    \draw[Black, thick] (9) -- (1);

    \end{tikzpicture}
    \caption{$\cg_2$: Uniform vertex colored odd cycle}\label{figure: uniform vertex odd cycle}
\end{minipage}
\end{figure}
\end{example}



We now extend our computational study to uniform edge coloring. We obtain the following counterexample which is an even cycle with uniform edge coloring.

\begin{example}\label{example:CounterexampleConjEdges}
Let $\cg$ be the colored $6$-cycle as shown in Figure \ref{figure:uniform edge even cycle}. Here, the computation gives us that the linear binomial $\sigma_{15}-\sigma_{24}$ lies in $I_\cg$ and that there are no graph symmetries in $\cg$. This is evident from the fact that there are exactly two red vertices in $\cg$ ($1$ and $2$) but they are adjacent to vertices of different colors ($6$ and $3$).
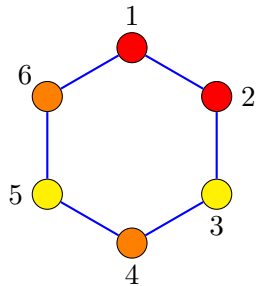
\begin{figure}[H]
    \centering
    \begin{tikzpicture}[scale=0.65, every node/.style={circle, draw, inner sep=1pt, minimum size=4mm, font=\small}]
        \node[fill=red, label=above:{1}] (1) at (90:2) {};
        \node[fill=red, label=right:{2}] (2) at (30:2) {};
        \node[fill=yellow, label=below:{3}] (3) at (330:2) {};
        \node[fill=orange, label=below:{4}] (4) at (270:2) {};
        \node[fill=yellow, label=left:{5}] (5) at (210:2) {};
        \node[fill=orange, label=above left:{6}] (6) at (150:2) {};
    
        \draw[blue, thick] (1) -- (2);
        \draw[blue, thick] (2) -- (3);
        \draw[blue, thick] (3) -- (4);
        \draw[blue, thick] (4) -- (5);
        \draw[blue, thick] (5) -- (6);
        \draw[blue, thick] (6) -- (1);
    \end{tikzpicture}
    \caption{Uniform edge colored even cycle}\label{figure:uniform edge even cycle}
\end{figure}
\end{example}

Note: The above example is also the first example where $\sigma_{ij}-\sigma_{xy} \in I_\cg$ but the sets $\{\lambda(i),\lambda(j)\}$ and $\{\lambda(x),\lambda(y)\}$ are different. In particular, $\det(K_{1\leftrightarrow 5\setminus \{1,5\}}) \neq \det(K_{2\leftrightarrow 4\setminus \{2,4\}})$ and $\det(K_{1\overset{c}\leftrightarrow 5\setminus \{1,5\}}) \neq \det(K_{2\overset{c}\leftrightarrow 4\setminus \{2,4\}})$ even though $\sigma_{15}-\sigma_{24}\in I_\cg$. 

Although we obtained a counterexample which is an even cycle with uniform edge coloring, our computational study failed to obtain a similar counterexample which is an odd cycle with uniform edge coloring. Thus, in the following subsection, we prove the revised version of the original conjecture.

\subsection{The revised conjecture}
In this subsection, we provide a revised version of the conjecture, namely, for uniform edge colored odd cycles. We also state the necessary results that are needed to eventually prove the statement. The revised conjecture is now as follows:
\begin{theorem}[The revised conjecture]\label{theorem:revisedConjecture}
Let $\cg$ be any colored $n$-cycle with uniform edge coloring, where $n$ is odd. Then a linear binomial lies in $I_\cg$ if and only if there is a corresponding symmetry in $\cg$. 
\end{theorem}

In order to prove this theorem, we first look at some necessary results based on the parity of the cycle size and path length. Recall that in Example \ref{example:CounterexampleConjEdges}, we saw that $\sigma_{15}-\sigma_{24} \in I_\cg$ even though $\det(K_{1\leftrightarrow 5\setminus \{1,5\}}) \neq \det(K_{2\leftrightarrow 4\setminus \{2,4\}})$ and $\det(K_{1\overset{c}\leftrightarrow 5\setminus \{1,5\}}) \neq \det(K_{2\overset{c}\leftrightarrow 4\setminus \{2,4\}})$. This can happen when there are some terms that get canceled between $k_{12}^2\det(K_{1\leftrightarrow 5\setminus \{1,5\}})$ and $k_{12}^2\det(K_{1\overset{c}\leftrightarrow 5\setminus \{1,5\}})$ and similarly, $k_{12}^2\det(K_{2\leftrightarrow 4\setminus \{2,4\}})$ and  $k_{12}^2\det(K_{2\overset{c}\leftrightarrow 4\setminus \{2,4\}})$. In particular, we have 
\begin{align*}
        \sigma_{15} 
        &= \frac{1}{\operatorname{det}(K)} ( k_{12}^2 ( k_{11} k_{33} k_{44} - k_{12}^2 k_{44} - k_{12}^2 k_{11}) + k_{12}^4 k_{44}), \ \text{and}\\
        \sigma_{24} 
        &= \frac{1}{\operatorname{det}(K)} ( k_{12}^2 ( k_{11} k_{33} k_{44} - k_{12}^2 k_{11} - k_{12}^2 k_{33}) + k_{12}^4 k_{33}).
    \end{align*}
Here, the polynomial corresponding to the shorter paths are given by \( k_{12}^2 ( k_{11} k_{33} k_{44} - k_{12}^2 k_{44} - k_{12}^2 k_{11}) \) and \( k_{12}^2 ( k_{11} k_{33} k_{44} - k_{12}^2 k_{11} - k_{12}^2 k_{33}) \), respectively, and the complementary paths are given by \(k_{12}^4 k_{44}\) and \(k_{12}^4 k_{33}\). It is clear that the term $k_{12}^4k_{44}$ gets canceled in $\sigma_{15}$, and similarly, $k_{12}^4k_{33}$ gets canceled in $\sigma_{24}$. 

In general, the cancellation of terms within the image of $\sigma_{ij}$ poses a significant challenge to analyze the linear binomials. However, this phenomenon cannot happen in the image of $\sigma_{ij}$ in an odd cycle, which we prove using the following lemma:

\begin{lemma}
\label{lem:VertexCoeffOddEvenPath}
Let $P$ be a path graph with $m$ vertices.
\begin{enumerate}
    \item If $m$ is odd, then every monomial in the expansion of \(\operatorname{det}(K_P)\) has an odd vertex degree, i.e., contains an odd number of diagonal entries.
    \item Similarly, if $m$ is even, then every monomial in the expansion of \(\operatorname{det}(K_P)\) has an even vertex degree.
\end{enumerate}
\end{lemma}
\begin{proof}
\textbf{1)} We apply the Leibniz formula to calculate the determinant of $K_P = (p_{ij})$. As $K_P$ is tridiagonal, any permutation $\tau$ that maps $i$ to an index not adjacent to $i$ results in a product term that is zero. Consequently, the only permutations that contribute to the determinant are either the identity or disjoint compositions of $2$-cycles, which permute adjacent indices. Let $\mathcal{T} \subseteq \mathbb{S}^m$ denote the set of such permutations. Thus, the Leibniz formula can be simplified to:
    \[
        \operatorname{det}(K_P)  
        = \sum_{\tau \in \mathcal{T}} \operatorname{sgn}(\tau) \prod_{\substack{i=1, \\ \tau(i) \neq i}}^m p_{i \tau(i)} \prod_{\substack{i=1, \\ \tau(i) = i}}^m p_{i \tau(i)}.
    \]
    Let $t_{\tau} \in \{0,1, \dots, \big\lfloor\frac{m}{2}\big \rfloor\}$ denote the number of 2-cycles within the permutation $\tau \in \mathcal{T}$. Each $\tau$ permutes an even amount of $2t_{\tau}$ off-diagonal entries and thus fixes $m-2t_{\tau}$ diagonal entries, which correspond to vertex coefficients. Since $m$ is odd and $2t_{\tau}$ is always even, the quantity $m-2t_{\tau}$ is odd. Hence, every nonzero monomial in the determinant contains an odd number of vertex coefficients.

\noindent \textbf{2)} An analogous argument applies in the even case.
\end{proof}

\begin{proposition}
\label{proposition:NoCancelInOddCycle}
Let \(\cg\) be a colored cycle of odd size \(n\). Then a linear binomial $\sigma_{ij}-\sigma_{xy}$ lies in the vanishing ideal if and only if the following conditions hold:
\begin{enumerate}
    \item $\det(K_{\setminus i\leftrightarrow j})=\det(K_{\setminus x\leftrightarrow y})$,
    \item $\det(K_{\setminus i\overset{c}\leftrightarrow j})=\det(K_{\setminus x\overset{c}\leftrightarrow y})$, and
    \item the multiset of edge colors in $i \leftrightarrow j$ and $x \leftrightarrow y$ and similarly $i \overset{c}\leftrightarrow j$ and $x \overset{c}\leftrightarrow y$ are equal.
\end{enumerate}
\end{proposition}
\begin{proof}
As the sufficient part is trivial, we focus on proving the necessary part of the proposition. Let \( \sigma_{ij} - \sigma_{xy}\) be a linear binomial that lies in $I_\cg$. 
As \(\cg\) is an odd cycle, the length of the shorter and complementary paths between any two vertices always have the opposite parity.
Without loss of generality, suppose the shorter paths \( i \leftrightarrow j\) and \(x \leftrightarrow y\) have even length. Then their corresponding complementary paths \( i \overset{c}{\leftrightarrow} j \) and \( x \overset{c}{\leftrightarrow} y \) must have odd length. Thus, \( K_{\setminus i \leftrightarrow j} \) and \( K_{\setminus x \leftrightarrow y} \) are concentration matrices of paths of odd length, while \( K_{\setminus i \overset{c}{\leftrightarrow} j} \) and \( K_{\setminus x \overset{c}{\leftrightarrow} y} \) are concentration matrices of paths of even length.
By Lemma~\ref{lem:VertexCoeffOddEvenPath}, we know that the determinants of the odd-length path matrices consist only of monomials with an odd number of vertex coefficients, whereas the determinants of the even-length path matrices consist only of monomials with an even number of vertex coefficients. Therefore, no monomial arising from a shorter path can cancel with any monomial arising from a complementary path, since they differ in the parity of the number of vertex coefficients. Hence, we can conclude that $\sigma_{ij}-\sigma_{xy}$ lies in $I_\cg$ only if their corresponding shorter path and complementary path polynomials match, respectively.
\end{proof}

As the proposition gives us a cancellation-free condition to obtain linear binomials in odd cycles, we now only need to analyze the determinants of uniform edge colored paths to prove Theorem \ref{theorem:revisedConjecture}. The following lemma gives us that no non-trivial path colorings can exist for paths $P$ and $Q$ to obtain the same determinant when they also satisfy uniform edge coloring.

\begin{lemma}\label{lemma:uniformEdgeNoNontrivial}
Let $P$ and $Q$ be two colored paths of length $m$ with uniform edge coloring. Then, $\det(K_P) = \det(K_Q)$ if and only if $P$ and $Q$ are either identical or reflection of each other.
\end{lemma}
\begin{proof}
Let the sequence of vertex colors of $P$ and $Q$ be $(p_1,p_2, \ldots, p_m)$ and $(q_1,q_2,\ldots, q_m)$, respectively. We denote the single edge color in $P$ and $Q$ by $e$. Now, let $P_i$ and $Q_i$ be the subpaths of $P$ and $Q$ obtained by taking the first $i$ vertices of $P$ and $Q$, respectively. By using the tridiagonal structure of the concentration matrix, the recurrence formula for the determinants can be written as
\[
\det(K_{P_i})= p_i \det(K_{P_{i-1}}) - e^2 \det(K_{P_{i-2}}) \text{ and } \det(K_{Q_i})= q_i \det(K_{Q_{i-1}}) - e^2 \det(K_{Q_{i-2}}), 
\]
with $\det(K_{P_1})=p_1$ and $\det(K_{Q_1})=q_1$.

Our goal now is to show that if the sequence $(p_1,p_2,\ldots,p_m)$ and $(q_1,q_2,\ldots, q_m)$ are not identical or reflection of each other, then $\det(K_P)$ is not equal to $\det(K_Q)$. So, let us assume that the two sequences differ for the first time at the $j$th position, i.e., $p_j\neq q_j$ (and also not equal to $q_{m-j+1}$) but $p_i=q_i$ for all $i <j$. This gives us that $\det(K_{P_i})=\det(K_{Q_i})$ for all $i<j$, and
\begin{eqnarray*}
&&\det(K_{P_j}) -\det(K_{Q_j})  \\
&&= p_j \det(K_{P_{j-1}}) - e^2 \det(K_{P_{j-2}}) - (q_j \det(K_{Q_{j-1}}) - e^2 \det(K_{Q_{j-2}})) \\
&&= (p_j-q_j)\det(K_{P_{j-1}}) \ \ (\text{As }\det(K_{P_{j-1}})=\det(K_{Q_{j-1}}) \text{ and } \det(K_{P_{j-2}})=\det(K_{P_{j-2}})) \\
&& \neq 0 \ \ (\text{As } p_j\neq q_j).
\end{eqnarray*}
So, considering $(\det(K_{P_1}), \det(K_{P_2}), \ldots, \det(K_{P_m}))$ and $(\det(K_{Q_1}), \det(K_{Q_2}), \ldots, \det(K_{Q_m}))$ as two sequences, the $j$th term becomes the first term where they differ. We claim that the entries of the two sequences are never equal for any $i\geq j$, which would eventually imply that $\det(K_P)\neq \det(K_Q)$. To show this, we compare the $j+1$th and $j+2$th terms of the two sequences. The difference of the $j+1$th terms is given by
\begin{eqnarray*}
&&\det(K_{P_{j+1}}) -\det(K_{Q_{j+1}})  \\
&&= p_{j+1} \det(K_{P_{j}}) - e^2 \det(K_{P_{j-1}}) - (q_{j+1} \det(K_{Q_{j}}) - e^2 \det(K_{Q_{j-1}})) \\
&&= p_{j+1}\det(K_{P_{j}}) - q_{j+1}\det(K_{Q_j}) \ \ (\text{As }\det(K_{P_{j-1}})=\det(K_{Q_{j-1}}) ).
\end{eqnarray*}
This gives us that regardless of whether $p_{j+1}$ is equal to $q_{j+1}$ or not, $\det(K_{P_{j+1}}) \neq \det(K_{Q_{j+1}})$. Similarly, we have
\begin{eqnarray*}
\det(K_{P_{j+2}}) &=& p_{j+2} \det(K_{P_{j+1}}) - e^2 \det(K_{P_{j}}) \\
&=& (p_{j+2}p_{j+1}p_j -p_{j+2}e^2-p_je^2)\det(K_{P_{j-1}}) - e^2(p_{j+2}p_{j+1} +1)\det(K_{P_{j-2}}), \text{ and} \\
\det(K_{Q_{j+2}})&=& (q_{j+2}q_{j+1}q_j -q_{j+2}e^2-q_je^2)\det(K_{Q_{j-1}}) - e^2(q_{j+2}q_{j+1} +1)\det(K_{Q_{j-2}}).
\end{eqnarray*}
This implies that $\det(K_{P_{j+2}})$ and $\det(K_{Q_{j+2}})$ can be equal only when the corresponding coefficients of $\det(K_{P_{j-1}})$ and $\det(K_{Q_{j-1}})$ are the same, as well as the ones of $\det(K_{P_{j-2}})$ and $\det(K_{Q_{j-2}})$. The only way that the coefficients can be equal is if the following equalities hold:
\[
p_{j}=q_{j+2}, p_{j+2}= q_{j}, p_{j+1}= q_{j+1}, \text{ and } q_j=q_{j+2}.
\]
However, the last equality gives us a contradiction as we had already assumed that $p_j\neq q_j$. This gives us that $\det(K_{P_{j+2}})\neq \det(K_{Q_{j+2}})$. Using a similar analysis and an inductive argument on the position of $j$, we can conclude that $\det(K_P)$ and $\det(K_Q)$ are not equal.
\end{proof}

We are now ready to prove the revised Conjecture \ref{theorem:revisedConjecture}.

\begin{proof}[Proof of Theorem \ref{theorem:revisedConjecture}]
Let $\cg$ be a colored cycle of odd length $n$ with uniform edge coloring. As the if-part of the statement follows from Proposition \ref{prop:symmetries}, we focus on proving the only if-part of the statement. Specifically, if $\sigma_{ij} - \sigma_{xy}$ lies in $I_{\cg}$, then we need to show that there exists a corresponding graph symmetry in $\cg$.
We use the indexing scheme illustrated in Figure~\ref{fig:IndexingUncoloredOddUni}. Note that the indices $i_l$ correspond to the vertices that are adjacent to $i,j, x$ and $y$.

Due to the uniform edge coloring of $\cg$,  we can write the covariance entries as follows by using Corollary~\ref{Cor:J&WforCycle}:
\begin{align*}
    &\sigma_{ij} = \frac{1}{\operatorname{det}(K_{\cg})} 
    \big( (-1)^{n_{i \leftrightarrow j}\;+\;1}  \ \ k_{12}^{n_{i \leftrightarrow j}\;-\;1} \ \ \operatorname{det}\big(K_{i_1 i_2 \dots i_6}\big) \
    + \ (-1)^{n_{i \overset{c}{\leftrightarrow} j}\;+\;1} \ \ k_{12}^{n_{i \overset{c}{\leftrightarrow} j}\;-\;1} \ \ \operatorname{det}\big(K_{i_7 i_8}\big)\big),\\
    &\sigma_{xy} = \frac{1}{\operatorname{det}(K_{\cg})} 
    \big( (-1)^{n_{x \leftrightarrow y}\;+\;1}  \ \ k_{12}^{n_{x \leftrightarrow y}\;-\;1} \ \ \operatorname{det}\big(K_{i_5 i_6 \dots i_2}\big) \
    + \ (-1)^{n_{x \overset{c}{\leftrightarrow} y}\;+\;1} \ \ k_{12}^{n_{x \overset{c}{\leftrightarrow} y}\;-\;1} \ \ \operatorname{det}\big(K_{i_3 i_4}\big)\big).
\end{align*}
Since $\cg$ is an odd cycle, by Proposition~\ref{proposition:NoCancelInOddCycle} we know that 
$\sigma_{ij}-\sigma_{xy}$ lies in $I_\cg$ only if 
\[
\det\big(K_{i_1 i_2 \dots i_6}\big) = \det\big(K_{i_5 i_6 \dots i_2}\big), \text{ and }  \det\big(K_{i_7 i_8}\big) = \det\big(K_{i_3 i_4}\big).
\]
However, every path within $\cg$ also has a uniform edge coloring. Thus, by Lemma ~\ref{lemma:uniformEdgeNoNontrivial}, we know that such determinant equalities hold only when the respective paths are either identical or reflection of each other.  
This leads to two possible cases: the two paths are identical to each other or are reflections of each other. 
This is because if the path \(\{i_1 - i_2 - \dots - i_6\}\) is identical to \(\{i_5 - i_6 - \dots - i_2\}\) and \(\{i_7 - i_8\}\) is a reflection of \(\{i_3 - i_4\}\), then \(\{i_7 - i_8\}\) also becomes identical to \(\{i_3 - i_4\}\) due to the fact that \(\{i_7 - i_8\}\) and \(\{i_3 - i_4\}\) are contained in \(\{i_5 - i_6 - \dots - i_2\}\) and \(\{i_1 - i_2 - \dots - i_6\}\), respectively.

\noindent \textbf{Case I}: The paths \(\{i_1 - i_2 - \dots - i_6\}\) and \(\{i_7 - i_8\}\) are reflections  of \(\{i_5 - i_6 - \dots - i_2\}\) and \(\{i_3 - i_4\}\), respectively:

The relevant paths are given by:
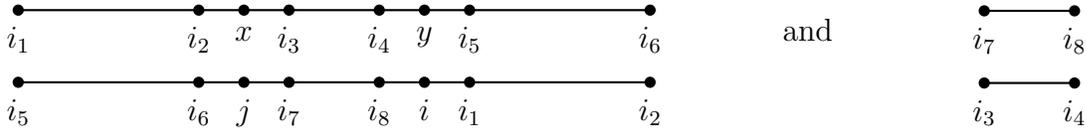
\begin{figure}[H]
    \centering
    \begin{minipage}{0.45\textwidth}
        \centering
        \begin{tikzpicture}[scale=1.2]

            \def\start{0}
            \def\bis{2}
            \def\small{0.5}

            \draw[thick] (\start, 0) -- ++(\start+2*\bis+6*\small, 0);

            \node[circle, fill=black, inner sep=1.5pt, label=below:{\(i_1\)}] at (\start, 0) {};
            \node[circle, fill=black, inner sep=1.5pt, label=below:{\(i_2\)}] at (\start+\bis, 0) {};
            \node[circle, fill=black, inner sep=1.5pt, label=below:{\(x\)}] at (\start+\bis+\small, 0) {};
            \node[circle, fill=black, inner sep=1.5pt, label=below:{\(i_3\)}] at (\start+\bis+2*\small, 0) {};
            \node[circle, fill=black, inner sep=1.5pt, label=below:{\(i_4\)}] at (\start+\bis+4*\small, 0) {};
            \node[circle, fill=black, inner sep=1.5pt, label=below:{\(y\)}] at (\start+\bis+5*\small, 0) {};
            \node[circle, fill=black, inner sep=1.5pt, label=below:{\(i_5\)}] at (\start+\bis+6*\small, 0) {};
            \node[circle, fill=black, inner sep=1.5pt, label=below:{\(i_6\)}] at (\start+2*\bis+6*\small, 0) {};

            \draw[thick] (\start, -0.8) -- ++(\start+2*\bis+6*\small, 0);

            \node[circle, fill=black, inner sep=1.5pt, label=below:{\(i_5\)}] at (\start, -0.8) {};
            \node[circle, fill=black, inner sep=1.5pt, label=below:{\(i_6\)}] at (\start+\bis, -0.8) {};
            \node[circle, fill=black, inner sep=1.5pt, label=below:{\(j\)}] at (\start+\bis+\small, -0.8) {};
            \node[circle, fill=black, inner sep=1.5pt, label=below:{\(i_7\)}] at (\start+\bis+2*\small, -0.8) {};
            \node[circle, fill=black, inner sep=1.5pt, label=below:{\(i_8\)}] at (\start+\bis+4*\small, -0.8) {};
            \node[circle, fill=black, inner sep=1.5pt, label=below:{\(i\)}] at (\start+\bis+5*\small, -0.8) {};
            \node[circle, fill=black, inner sep=1.5pt, label=below:{\(i_1\)}] at (\start+\bis+6*\small, -0.8) {};
            \node[circle, fill=black, inner sep=1.5pt, label=below:{\(i_2\)}] at (\start+2*\bis+6*\small, -0.8) {};

        \end{tikzpicture}
    \end{minipage}
    \hspace{2.75cm}
    \raisebox{0.5cm}{and} 
    \begin{minipage}{0.3\textwidth}
        \centering
        \begin{tikzpicture}[scale=1.2]

            \def\start{0}
            \def\small{0.5}

            \draw[thick] (\start, 0) -- ++(2*\small, 0);

            \node[circle, fill=black, inner sep=1.5pt, label=below:{\(i_7\)}] at (\start, 0) {};
            \node[circle, fill=black, inner sep=1.5pt, label=below:{\(i_8\)}] at (\start+2*\small, 0) {};

            \draw[thick] (\start, -0.8) -- ++(2*\small, 0);

            \node[circle, fill=black, inner sep=1.5pt, label=below:{\(i_3\)}] at (\start, -0.8) {};
            \node[circle, fill=black, inner sep=1.5pt, label=below:{\(i_4\)}] at (\start+2*\small, -0.8) {};

        \end{tikzpicture}
    \end{minipage}
\end{figure}
As the paths are reflection of each other, we get the following equalities of vertex colors:
\[
\lambda(i_1) = \lambda(i_2), \lambda(x)= \lambda(i), \lambda(i_3) = \lambda(i_8), \lambda(i_4) = \lambda(i_7), \lambda(y)= \lambda(j), \text{ and } \lambda(i_5)=\lambda(i_6).
\]
This configuration gives rise to the coloring pattern shown in Figure~\ref{fig:ColoredOddUni}. Using this color configuration, we can obtain the required graph symmetry (which is a reflection in this case) by the following analysis: As $\cg$ has odd length, one of the paths \(\{i_1 - i_2\}\) or \(\{i_5 - i_6\}\) is of odd length, while the other one is of even length. Thus, the graph symmetry corresponding to the linear binomial is the reflection along the axis that intersects the path of odd length at its middle vertex, and the path of even length at its middle edge. 

\noindent \textbf{Case II}: The paths \(\{i_1 - i_2 - \dots - i_6\}\) and \(\{i_7 - i_8\}\) are identical to \(\{i_5 - i_6 - \dots - i_2\}\) and \(\{i_3 - i_4\}\), respectively. This case follows using the similar argument as in Case I. 

As a similar argument would also follow for the linear binomials $\sigma_{ij}-\sigma_{xy}$ where $i \leftrightarrow j$ and $x\leftrightarrow y$ coincide or when $i=j$ and $x=y$, we can conclude that a linear binomial lies in $I_\cg$ for a uniform edge colored odd cycle if and only if there exists a corresponding symmetry in $\cg$.
\begin{figure}[h]
\begin{minipage}{0.45\textwidth}
\centering
\begin{tikzpicture}[scale=0.75]
    \def\r{2}

    \draw[thick, black] (0,0) circle (\r);

    \newcommand{\dotLabel}[3]{
        \node[circle, fill=black, inner sep=1.5pt, label=#3:{\(#2\)}] at ({#1}:{\r}) {};
    }
    
    \newcommand{\dotLabelShift}[5]{
    \node[
        circle, fill=black, inner sep=1.5pt, 
        label={[xshift=#3, yshift=#4]#2:{\(#1\)}}
        ] at (#2:\r) {};        
}

    \dotLabel{45}{i_2}{above right}
    \dotLabelShift{x}{30}{2pt}{-2pt}{black}
    \dotLabel{15}{i_3}{right}

    \dotLabel{135}{i_1}{above left}
    \dotLabelShift{i}{150}{-2pt}{-3pt}{black}
    \dotLabel{165}{i_8}{left}

    \dotLabel{225}{i_6}{below left}
    \dotLabelShift{j}{210}{-2pt}{2pt}{black}
    \dotLabel{195}{i_7}{left}

    \dotLabel{315}{i_5}{below right}
    \dotLabelShift{y}{330}{1pt}{3pt}{black}
    \dotLabel{345}{i_4}{right}
\end{tikzpicture}
\caption{}
\label{fig:IndexingUncoloredOddUni}
\end{minipage}%
\begin{minipage}{0.45\textwidth}
\centering
\begin{tikzpicture}[scale=0.75]
    \def\r{2}

    \draw[thick, black] (0,0) circle (\r);
    
    \draw[dashed, thick] (0, -\r - 0.3) -- (0, \r + 0.3);

    \newcommand{\dotLabel}[4]{
        \node[circle, fill=#4, inner sep=1.5pt, label=#3:{\(#2\)}] at ({#1}:{\r}) {};
    }

    \dotLabel{45}{i_2}{above right}{red}
    \dotLabel{30}{x}{above right}{yellow}
    \dotLabel{15}{i_3}{right}{purple}
    
    \dotLabel{135}{i_1}{above left}{red}
    \dotLabel{150}{i}{above left}{yellow}
    \dotLabel{165}{i_8}{left}{purple}
    
    \dotLabel{225}{i_6}{below left}{pink}
    \dotLabel{210}{j}{below left}{orange}
    \dotLabel{195}{i_7}{left}{magenta}
    
    \dotLabel{315}{i_5}{below right}{pink}
    \dotLabel{330}{y}{below right}{orange}
    \dotLabel{345}{i_4}{right}{magenta}
\end{tikzpicture}
\caption{}\label{fig:ColoredOddUni}
\end{minipage}
\end{figure}
\end{proof}

\section{Discussion and open problems}
\label{Sec:Discussion and open problems}
In Section \ref{section:counterexamples and strengthening} we constructed counterexamples to the general version of Conjecture \ref{Conj:ifandonlyif} by using the non-trivial path color configurations obtained in Section \ref{section:analysis of path graphs}. However, this also raises an interesting question of whether there are other non-trivial path color configurations that could give us the same determinant of the concentration matrices. Thus, we explore the other possible path color configurations in this section. To this end, we present some conditions that needs to be satisfied by any potential color configuration that could give us the same determinant. The overarching technique behind obtaining these conditions is based on comparing the monomials in the two determinants that have the same vertex and edge degree. Based on the conditions obtained, we conjecture that there are indeed no other non-trivial path color configurations apart from the ones stated in Theorem \ref{theorem:pathDetTypeEven} and \ref{theorem:PathDetTypeOdd}. We end the section by proposing the generalized version of the problem that could be of interest to the algebraic graph theory and combinatorics community, i.e., when can the determinant of the concentration matrices of two arbitrary colored graphs be equal. We provide an example of two such colored graphs that have the same determinant, which could be a starting point to further explore the problem.

\begin{lemma}\label{lemma:equal determinant conditions} 
Let \( P \) and \( Q \) be two colored paths with $\det(K_P)=\det(K_Q)$. Then the following conditions hold:
\begin{enumerate}
\item The multiset of vertex and edge colors used in \(P\) is the same as that in \( Q \).
\item The vertex colors adjacent to each edge color remain the same in both paths.
\end{enumerate}
\end{lemma}
\begin{proof}
We denote the entries of $K_P$ by $p_{ij}$ and those of of $K_Q$ by $q_{ij}$.

\noindent\textbf{1)} By Lemma \ref{lem:tridiagonal det} we know that there exist the monomial $\prod_{i=1}^m p_{ii}$ in $\det(K_P)$ and similarly, $\prod_{i=1}^m q_{ii}$ in $\det(K_Q)$ corresponding to $S=\emptyset$ in the sum. However, as these are the only monomials in the determinant that have vertex degree $m$ and edge degree $0$, they need to be equal to each other. Thus, the multiset of vertex colors used in $P$ and $Q$ are the same. Similarly, for every edge $\{i,j\}$, notice that there exists a monomial $p_{ij}^2\prod_{v\in V_P\setminus \{i,j\}} p_{vv}$ in $\det(K_P)$ and similarly, $q_{ij}^2\prod_{v\in V_P\setminus \{i,j\}} q_{vv}$ in $\det(K_Q)$. As $\det(K_P)$ and $\det(K_Q)$ are equal as polynomials, we know that the collection of such monomials corresponding to any edge color has to match. Thus, we can conclude that the multiset of edge colors in $P$ and $Q$ also has to be equal. 


\noindent \textbf{2)} There can exist multiple edges of the same color. However, from (1), we know that the multiset of edge colors used in $P$ and $Q$ remains the same. This implies that for any edge $\{a,b\}$ in $P$ of a given edge color, there must exist an edge $\{x,y\}$ in $Q$ of the same color such that the corresponding monomials match. This gives us that 
\[
p_{ab}^2\prod_{v\in V(P)\setminus \{a,b\}}p_{vv}= q_{xy}^2\prod_{v\in V(Q)\setminus \{x,y\}}q_{vv},
\]
implying that $\prod_{v\in V(P)\setminus \{a,b\}}p_{vv}= \prod_{v\in V(Q)\setminus \{x,y\}}q_{vv}$ as $\{a,b\}$ and $\{x,y\}$ have the same edge color. As we also know from 1) that $\prod_{i=1}^m p_{ii}=\prod_{i=1}^m q_{ii}$, we can conclude that $p_{aa}p_{bb}$ is equal to $p_{xx}p_{yy}$. Thus, we can conclude that for a given edge color, the vertex colors adjacent to each edge of that color remains the same in $P$ and $Q$. 
\end{proof}

We now state a coloring condition which is dependent on the parity of the number of vertices in the path.

\begin{lemma}
\label{lem:OddEvenPathVertexColorSets}
Let \( P \) and \( Q \) be two colored paths on $m$ vertices with $\det(K_P)=\det(K_Q)$. Then the following conditions hold:
\begin{enumerate}
    \item If $m$ is odd, then the multiset of odd vertex colors (and similarly even vertex colors) used in \(P\) and \(Q\) remains the same.
    \item If $m$ is even, then the multiset of odd edge colors (and similarly even edge colors) used in \(P\) and \(Q\) remains the same.
\end{enumerate}
\end{lemma}

\begin{proof}
 \textbf{1)} When $m$ is odd, we first pick an odd vertex $a$. By Lemma \ref{lem:tridiagonal det}, we consider the monomial in the sum corresponding to the set $S$, where $S$ is the collection of all edges that are not adjacent to $a$. This monomial has the following structure:
\[
p_{12}^2p_{34}^2\ldots p_{a-2a-1}^2p_{aa}p_{a+1a+2}^2\ldots p_{m-1m}^2. 
\]
It is clear that these are the only monomials in $\det(K_P)$ with vertex degree one and edge degree $(m-1)/2$, and the corresponding vertex color $p_{aa}$ can only be of an odd vertex. As the set of these monomials has to match in $\det(K_P)$ and $\det(K_Q)$, we can conclude that the multiset of odd vertex colors remains the same in $P$ and $Q$. Furthermore, since we already know from Lemma \ref{lemma:equal determinant conditions} that the multiset of all vertex colors also remain the same, we can also conclude that the multiset of even vertex colors remains the same in $P$ and $Q$.  

\noindent \textbf{2)}
When $m$ is even, using Lemma \ref{lem:tridiagonal det}, we consider the monomial in the sum corresponding to the set of all odd edges. (Here, by odd edges we mean edges of the form $\{i,i+1\}$ where $i$ is odd.) This turns out to be the unique monomial  $p_{12}^2p_{34}^2\ldots p_{m-1m}^2$ in $\det(K_P)$ with vertex degree zero and edge degree $m/2$. As this monomial has to match with the corresponding monomial in $\det(K_Q)$, we can conclude that the muliset of odd edge colors remains the same in $P$ and $Q$. Combining this result with Lemma \ref{lemma:equal determinant conditions}, we can also conclude that the multiset of even edge colors remains the same in $P$ and $Q$.   
\end{proof}

Based on the conditions obtained in Lemma \ref{lemma:equal determinant conditions} and \ref{lem:OddEvenPathVertexColorSets}, we believe that the only possible non trivial color configurations of $P$ and $Q$ such that $\det(K_P)$ is equal to $\det(K_Q)$, are the ones obtained in Theorem \ref{theorem:pathDetTypeEven} and \ref{theorem:PathDetTypeOdd} and their reflections. Thus, we state the following conjecture:

\begin{conjecture}
Let $P$ and $Q$ be two colored paths on $m$ vertices with $\det(K_P)=\det(K_Q)$. 
\begin{enumerate}
    \item If $m$ is even, then one of the following conditions holds:
    \begin{itemize}
        \item $Q$ is identical to $P$, 
        \item $Q$ is a reflection of $P$, 
        \item $P$ and $Q$ satisfy the color configuration stated in Theorem \ref{theorem:pathDetTypeEven} or is a reflection of the same.
\end{itemize}
\item Similarly, if $m$ is odd, then one of the following conditions holds:
\begin{itemize}
        \item $Q$ is identical to $P$, 
        \item $Q$ is a reflection of $P$, 
        \item $P$ and $Q$ satisfy the color configuration stated in Theorem \ref{theorem:PathDetTypeOdd} or is a reflection of the same.
\end{itemize}
\end{enumerate}
\end{conjecture}

This brings us to pose the generalized version of the problem, which is the following:
\begin{problem}
Let $\cg_1$ and $\cg_2$ be two arbitrary colored graphs. Then what are the structural and coloring conditions that $\cg_1$ and $\cg_2$ need to satisfy such that $\det(K_{\cg_1})=\det(K_{\cg_2})$?   
\end{problem}

To the best of our knowledge, this problem has not been studied before. The only related study done in this direction is in \cite{LiptonBishnoi2004}, where the authors study the conditions on when the generalized characteristic polynomial of two graphs are equal. For a given graph $G$ with adjacency matrix $A$, the characteristic polynomial of $G$ is defined as $\det(A-\lambda I)$. The generalization of the characteristic polynomial is denoted by $\mathcal{L}_G(x,y,\lambda)$ and is obtained by replacing $A$ with $A(x,y)$, where the $1$s are replaced by $x$ and the $0$s by $y$. The authors proved in \cite{LiptonBishnoi2004} that if two graphs are \textit{co-spectral} and their complements are also co-spectral, then they have the same $\mathcal{L}$-polynomial. Interestingly, there is a subtle connection between the $\mathcal{L}$-polynomial and the determinant of the concentration matrix of $\cg$. Specifically, if $\cg$ has uniform coloring, then it is easy to see that $\det(K_\cg)$ is equal to $\mathcal{L}_G(x,0,\lambda)$. Here, $x$ corresponds to the edge color and $\lambda$ corresponds to the vertex color. However, the problem that we present is not necessarily confined to uniform coloring, and hence cannot be solved by the co-spectral property. 

By performing a computational study, we obtained the following example of two uniform vertex colored graphs and two uniform edge colored graphs whose concentration matrices have the same determinant.

\begin{example}
Let $\cg_1$ and $\cg_2$ be the two uniform vertex colored graphs shown in Figure \ref{figure:2 uniform vertex co-spectral}. Computing the determinant of $K_{\cg_1}$ and $K_{\cg_2}$ gives us that
\begin{eqnarray*}
&&\det(K_{\cg_1})=\det(K_{\cg_2}) = \\
&&\textcolor{red}{k_{11}}^6 
+ 6\,\textcolor{red}{k_{11}}^2 \textcolor{blue}{k_{12}}^4 
- 6\,\textcolor{red}{k_{11}}^4 \textcolor{blue}{k_{12}}^2 
- \textcolor{red}{k_{11}}^4 \textcolor{green}{k_{23}}^2 
- \textcolor{blue}{k_{12}}^4 \textcolor{green}{k_{23}}^2 
+ 3\,\textcolor{red}{k_{11}}^2 \textcolor{blue}{k_{12}}^2 \textcolor{green}{k_{23}}^2 
- 2\,\textcolor{red}{k_{11}}^2 \textcolor{blue}{k_{12}}^3 \textcolor{green}{k_{23}}.
\end{eqnarray*}
Similarly, let $\cg_3$ and $\cg_4$ be the uniform edge colored graphs shown in Figure \ref{figure:2 uniform edge co-spectral}. Computing the determinant of $K_{\cg_3}$ and $K_{\cg_4}$ gives us that
\begin{eqnarray*}
&&\det(K_{\cg_3})=\det(K_{\cg_4}) = \\
&& \textcolor{yellow}{k_{22}}^5 \textcolor{red}{k_{11}}
- \textcolor{blue}{k_{12}}^6
+ 5\,\textcolor{yellow}{k_{22}}^2 \textcolor{blue}{k_{12}}^4
+ 2\,\textcolor{yellow}{k_{22}} \textcolor{red}{k_{11}} \textcolor{blue}{k_{12}}^4
- 2\,\textcolor{yellow}{k_{22}}^4 \textcolor{blue}{k_{12}}^2
- 5\,\textcolor{yellow}{k_{22}}^3 \textcolor{red}{k_{11}} \textcolor{blue}{k_{12}}^2.
\end{eqnarray*}
\begin{figure}
\begin{center}
\begin{tikzpicture}[scale=0.65]
        \node[draw, circle, fill=red] (1) at (0, 2) {};
        \node[draw, circle, fill=red] (2) at (2, 2) {};
        \node[draw, circle, fill=red] (3) at (4, 2) {};
        \node[draw, circle, fill=red] (4) at (4, 0) {};
        \node[draw, circle, fill=red] (5) at (2, 0) {};
        \node[draw, circle, fill=red] (6) at (0, 0) {};
        
        \draw[thick, blue] (1) -- (2);
        \draw[thick, blue] (2) -- (3);
        \draw[thick, blue] (2) -- (5);
        \draw[thick, blue] (3) -- (4);
        \draw[thick, green] (4) -- (5);
        \draw[thick, blue] (5) -- (6);
        \draw[thick, blue] (6) -- (1);
    
        \node at (-0.5, 2.6) {1};
        \node at (2, 2.6) {2};
        \node at (4.5,2.6) {3};
        \node at (4.5, -0.6) {4};
        \node at (2, -0.6) {5};
        \node at (-0.5, -0.6) {6};

    \end{tikzpicture}
    \hspace{1cm}
    \centering
    \begin{tikzpicture}[scale=0.65]
        \node[draw, circle, fill=red] (1) at (0, 2) {};
        \node[draw, circle, fill=red] (2) at (2, 2) {};
        \node[draw, circle, fill=red] (3) at (4, 2) {};
        \node[draw, circle, fill=red] (4) at (4, 0) {};
        \node[draw, circle, fill=red] (5) at (2, 0) {};
        \node[draw, circle, fill=red] (6) at (0, 0) {};
        
        \draw[thick, blue] (1) -- (2);
        \draw[thick, blue] (2) -- (3);
        \draw[thick, blue] (3) -- (6);
        \draw[thick, blue] (3) -- (4);
        \draw[thick, green] (4) -- (5);
        \draw[thick, blue] (5) -- (6);
        \draw[thick, blue] (4) -- (1);
    
        \node at (-0.5, 2.6) {1};
        \node at (2, 2.6) {2};
        \node at (4.5,2.6) {3};
        \node at (4.5, -0.6) {4};
        \node at (2, -0.6) {5};
        \node at (-0.5, -0.6) {6};
\end{tikzpicture}
\caption{$\cg_1$ and $\cg_2$}\label{figure:2 uniform vertex co-spectral}
\end{center}

\begin{center}
\begin{tikzpicture}[scale=0.65]
        \node[draw, circle, fill=red] (1) at (0, 2) {};
        \node[draw, circle, fill=yellow] (2) at (2, 2) {};
        \node[draw, circle, fill=yellow] (3) at (4, 2) {};
        \node[draw, circle, fill=yellow] (4) at (4, 0) {};
        \node[draw, circle, fill=yellow] (5) at (2, 0) {};
        \node[draw, circle, fill=yellow] (6) at (0, 0) {};
        
        \draw[thick, blue] (1) -- (2);
        \draw[thick, blue] (2) -- (3);
        \draw[thick, blue] (2) -- (5);
        \draw[thick, blue] (3) -- (4);
        \draw[thick, blue] (4) -- (5);
        \draw[thick, blue] (5) -- (6);
        \draw[thick, blue] (6) -- (1);
    
        \node at (-0.5, 2.6) {1};
        \node at (2, 2.6) {2};
        \node at (4.5,2.6) {3};
        \node at (4.5, -0.6) {4};
        \node at (2, -0.6) {5};
        \node at (-0.5, -0.6) {6};

    \end{tikzpicture}
    \hspace{1cm}
    \centering
    \begin{tikzpicture}[scale=0.65]
        \node[draw, circle, fill=red] (1) at (0, 2) {};
        \node[draw, circle, fill=yellow] (2) at (2, 2) {};
        \node[draw, circle, fill=yellow] (3) at (4, 2) {};
        \node[draw, circle, fill=yellow] (4) at (4, 0) {};
        \node[draw, circle, fill=yellow] (5) at (2, 0) {};
        \node[draw, circle, fill=yellow] (6) at (0, 0) {};
        
        \draw[thick, blue] (1) -- (2);
        \draw[thick, blue] (2) -- (3);
        \draw[thick, blue] (3) -- (6);
        \draw[thick, blue] (3) -- (4);
        \draw[thick, blue] (4) -- (5);
        \draw[thick, blue] (5) -- (6);
        \draw[thick, blue] (4) -- (1);
    
        \node at (-0.5, 2.6) {1};
        \node at (2, 2.6) {2};
        \node at (4.5,2.6) {3};
        \node at (4.5, -0.6) {4};
        \node at (2, -0.6) {5};
        \node at (-0.5, -0.6) {6};
\end{tikzpicture}
\caption{$\cg_3$ and $\cg_4$}\label{figure:2 uniform edge co-spectral}
\end{center}
\end{figure}
\end{example}

\section*{Acknowledgements}
The authors are grateful for helpful discussions with Mathias Drton, Benjamin Hollering, Tobias Boege and Jane Ivy Coons. Both authors received funding from the European Research Council (ERC) under the European Union’s Horizon 2020 research and innovation programme (grant agreement No. 883818).

\printbibliography
\end{document}